\documentclass{article}

\usepackage{hyperref}
\usepackage{amsmath}
\usepackage{amssymb}
\usepackage{amsthm}
\usepackage{bbm,mathtools}
\usepackage{enumitem}
\usepackage{enumitem}
\usepackage{theoremref}

\newcounter{dummy}
\newcommand\myitem[1][]{\item[#1]\refstepcounter{dummy}\def\@currentlabel{#1}}

\newcommand{\midb}{\;\middle|\;}
\newcommand{\one}{\mathbbm 1}

\def\argmax{\mathop{\rm argmax}\limits}

\def\reals{\mathbb{R}}

\def\uball{\mathbb{B}}
\def\ereals{\overline{\mathbb{R}}}

\def\comp{\raise 1pt \hbox{$\scriptstyle\circ$}}

\def\argmin{\mathop{\rm argmin}\limits}

\def\minimize{\mathop{\rm minimize}\limits}
\def\maximize{\mathop{\rm maximize}\limits}

\def\st{\mathop{\rm subject\ to}}

\def\dom{\mathop{\rm dom}\nolimits}

\def\upto{{\raise 1pt \hbox{$\scriptstyle \,\nearrow\,$}}}
\def\downto{{\raise 1pt \hbox{$\scriptstyle \,\searrow\,$}}}

\def\cl{\mathop{\rm cl}\nolimits}

\def\epi{\mathop{\rm epi}}

\def\FF{(\F_t)_{t\ge 0}}
\def\ovr{\mathop{\rm over}}

\def\A{{\cal A}}
\def\B{{\cal B}}

\def\F{{\cal F}}

\def\G{{\cal G}}
\def\H{{\cal H}}

\def\L{{\cal L}}

\def\N{{\cal N}}

\def\T{{\cal T}}

\newtheorem{theorem}{Theorem}
\newtheorem{lemma}[theorem]{Lemma}

\newtheorem{example}[theorem]{Example}
\newtheorem{cexample}[theorem]{Counterexample}
\newtheorem{remark}[theorem]{Remark}
\theoremstyle{definition}
\newtheorem{definition}{Definition}
\newtheorem{assumption}{Assumption}
\theoremstyle{empty}

\title{Dynamic programming in convex stochastic optimization}

\author{Teemu Pennanen\thanks{Department of Mathematics, King's College London, Strand, London, WC2R 2LS, United Kingdom, teemu.pennanen@kcl.ac.uk} \and Ari-Pekka Perkki\"o\thanks{Mathematics Institute, Ludwig-Maximilian University of Munich, Theresienstr. 39, 80333 Munich, Germany, a.perkkioe@lmu.de. Corresponding author}}

\begin{document}

\maketitle

\begin{abstract}
This paper studies the dynamic programming principle for general convex stochastic optimization problems introduced by Rockafellar and Wets in \cite{rw76}. We extend the applicability of the theory by relaxing compactness and boundedness assumptions. In the context of financial mathematics, the relaxed assumption are satisfied under the well-known no-arbitrage condition and the reasonable asymptotic elasticity condition of the utility function. Besides financial mathematics, we obtain several new results in linear and nonlinear stochastic programming and stochastic optimal control.
\end{abstract}

\newline
\newline

\section{Introduction}

Given a probability space $(\Omega,\F,P)$ with a filtration $\FF$ (an increasing sequence of sub-$\sigma$-algebras of $\F$), consider the convex stochastic optimization problem
\begin{equation}\label{sp}\tag{$SP$}
\minimize\quad Eh(x)\quad\ovr x\in\N,
\end{equation}
where $\N$ is a linear space of stochastic processes $x=(x_t)_{t=0}^T$ adapted to $\FF$ (i.e., $x_t$ is $\F_t$-measurable). We assume that $x_t$ takes values in a Euclidean space $\reals^{n_t}$ so the process $x=(x_t)_{t=0}^T$ takes values in $\reals^n$ where $n:=n_0+\cdots+n_T$. The objective is defined on the space $L^0(\Omega,\F,P;\reals^n)$ of $\F$-measurable $\reals^n$-valued functions $x$ by
\[
Eh(x) := \int_\Omega h(x(\omega),\omega)dP(\omega),
\]
where $h$ is a convex normal integrand, i.e.\ an extended real-valued function on $\reals^n\times\Omega$ such that $\omega\mapsto\epi h(\cdot,\omega)$ is a closed convex-valued measurable mapping; see e.g.\ \cite[Chapter~14]{rw98}. By \cite[Proposition~14.28]{rw98}, $\omega\mapsto h(x(\omega),\omega)$ is measurable for all $x\in L^0(\Omega,\F,P;\reals^n)$. Here and in what follows, we define the integral of an extended real-valued random variable as $+\infty$ unless its positive part is integrable. The integral functional $Eh$ is thus a well-defined extended real-valued convex function on $L^0(\Omega,\F,P;\reals^n)$.

Problems of the form \eqref{sp} were first studied in \cite{rw76,evs76} where it was observed that many more specific stochastic optimization problems can be written in this unified format. These include more traditional formulations of stochastic programming, convex stochastic control and various problems in financial mathematics; see Section~\ref{sec:dpapp} below. 

The articles \cite{rw76,evs76} gave formulations of the {\em dynamic programming} principle general enough to apply to the abstract stochastic optimization format \eqref{sp}. They extend many earlier formulations of the stochastic dynamic programming principle such as those in \cite{dan55,bea55,bel57,dyn72}. Like the present article, \cite{rw76} studied the convex case where decisions are described by finite dimensional vectors. The article \cite{evs76} extended the results to nonconvex problems where decisions are described by elements of Polish spaces. Both assumed that the objective is lower bounded and that the decisions are taken from a compact set, uniformly compact in \cite{rw76}. The compactness assumptions were relaxed in \cite{pp12} in the convex case and in \cite{ppr16} in the nonconvex case. This paper relaxes the lower boundedness assumption of the objective. This is interesting in many applications in financial mathematics as well as in mathematical programming such as linear stochastic optimization. In portfolio optimization problems, the relaxed conditions holds if the utility function satisfies the {\em reasonable asymptotic elasticity condition} that is extensively studied in financial mathematics; see e.g.\ \cite{ks99,rs5}. These problems are often set in a stochastic control format, which motivates the study of stochastic control beyond lower bounded objectives as well. We also provide new results on the theory of normal integrands concerning conditional independence. This allows us to derive many well-known results on Markov decision processes using the theory of normal integrands; see Section~\ref{sec:dpapp}. 

Like \cite{rw76,evs76}, our approach builds on the notion of {\em conditional expectation of a normal integrand}, introduced in \cite{bis73}. This allows for significant extensions to many better known formulations of stochastic dynamic programming while greatly simplifying the measurability questions that may come up in the dynamic programming recursion. More traditional formulations will be obtained as special cases in Section~\ref{sec:dpapp} below.  A good illustration of potential measurability complications can be found in \cite{bs78} which follows a different line of analysis not building on the theory of normal integrands; see \cite[Section 1.2.II]{bs78} and Section~\ref{sec:ocdp} below for a further comparison.


An extended real-valued random variable $X$ is said to be {\em quasi-integrable} if either $X^+$ or $X^-$ is integrable. Given a quasi-integrable $X$ and a $\sigma$-algebra $\G \subseteq \F$, there exists an extended real-valued $\G$-measurable random variable $E^\G X$, almost surely unique, such that
\begin{align*}
E\left[\alpha(E^\G X)\right] = E\left[\alpha X\right] \quad\forall \alpha\in L^\infty_+(\Omega,\G,P).
\end{align*}
The random variable $E^\G{X}$ is known as the $\G$-conditional expectation  of $X$. The following extends the operation of conditional expectation to normal integrands. Sufficient conditions for its existence will be given in Section~\ref{sec:ceni}.

\begin{definition}
We say that a normal integrand $E^\G h$ is the {\em $\G$-conditional expectation} of a normal integrand $h$ if it is the almost surely everywhere unique $\G$-measurable normal integrand such that
\[
(E^\G h)(x)=E^\G[h(x)]\quad\text{a.s.}
\]
for all $x\in L^0(\G)$ for which $h(x)$ is quasi-integrable.
\end{definition}

We will use the notations $x^t=(x_0,\ldots,x_t)$, $n^t=n_0+\cdots+n_t$ and $E_t=E^{\F_t}$. We say that an adapted sequence $(h_t)_{t=0}^T$ of normal integrands $h_t:\reals^{n^t}\times\Omega\rightarrow\ereals$ solves the generalized {\em Bellman equations for $h$} if
\begin{equation}\tag{BE}\label{be}
\begin{split}
h_T &= E_Th,\\
h_t &= E_t\inf_{x_{t+1}} h_{t+1}\quad t=T-1,\ldots,0.
\end{split}
\end{equation}
More precisely, this means that there exists another sequence $(\tilde h_t)_{t=0}^{T-1}$ of normal integrands such that
\begin{equation}\tag{BE}\label{be}
\begin{split}
h_T &=E_Th,\\
\tilde h_t(x^t,\omega)&=\inf_{x_{t+1}\in\reals^{n_{t+1}}}h_{t+1}(x^t,x_{t+1},\omega),\\
h_t &= E_t\tilde h_t
\end{split}
\end{equation}
for $t=T-1,\ldots,0$. This paper gives sufficient conditions for the existence of the solutions of \eqref{be}; see \thref{thm:belb,thm:be}  below. We also show that, when solutions exist, they provide useful characterizations of the optimum values and solutions of \eqref{sp}; see \thref{thm:dp0,thm:dpLb} below.  In the literature of stochastic control, results such as \thref{thm:dp0,thm:dpLb} relating the solutions of the Bellman equations to the solutions of the optimization problem are often called ``verification theorems''; see e.g., \cite{fr75}. \thref{thm:dp0,thm:dpLb} are illustrated by deriving several new results, e.g., in linear stochastic programming, stochastic control and financial mathematics.

\section{Conditional expectations of normal integrands}\label{sec:ceni}

The general theory of dynamic programming studied in this paper builds on conditional expectations of normal integrands. This section reviews the theory that will be used in the analysis in the subsequent sections. Most of the results below can be found in the literature but we include the simple proofs because that allows us to make some extensions to the existing theory.


\subsection{Existence}\label{ssec:euni}

Conditional expectations of convex normal integrands were introduced in \cite{bis73}. The article \cite{de76} extended the definition to general $\B\otimes\F$-measurable lower bounded integrands. More general conditions for the existence of a conditional expectation of a normal integrand have been given in \cite{thi81,tru91,chs3}. Our arguments and conditions for existence in Section~\ref{ssec:euni} are largely from \cite{bis73}.

We say that a normal integrand $h$ is {\em L-bounded} if there exist $\rho,m\in L^1$ with
\[
h(x) \ge -\rho|x|-m \quad\forall x\in\reals^n.
\]
Condition 3 in the following lemma was used for the existence of a conditional expectation in \cite[Theorem~2]{bis73}.

\begin{lemma}\thlabel{lem:lbconv}
For a convex normal integrand $h$, the following are equivalent:
\begin{enumerate}
\item
  $h$ is L-bounded,
\item
  there exist $v\in L^1(\reals^n)$ and $m\in L^1$ such that
  \[
  h(x,\omega)\ge x\cdot v(\omega)-m(\omega),
  \]
\item
  $\dom Eh^*\cap L^1\ne\emptyset$.  
\end{enumerate}
\end{lemma}

\begin{proof}
Let $v\in L^1$ such that $Eh^*(v)<\infty$ is finite. By Fenchel's inequality,
\[
h(x,\omega)\ge x\cdot v - h^*(v,\omega) \ge -|v||x|-h^*(v,\omega),
\]
so we may choose $\rho=|v|$ and $m=h^*(v)^+$. On the other hand, $h\ge-\rho|\cdot|-m$ can be written as $(h+\rho|\cdot|)^*(0)\le m$. By \cite[Theorem~16.4]{roc70a}, this means that
\[
\inf_{v\in\reals^n}\{h^*(v)+\delta_\uball(v/\rho)\}\le m,
\]
where the infimum is attained. By \cite[Theorem~14.37]{rw98}, there is a $v\in L^0$ with $|v|\le\rho$ and $h^*(v)\le m$.
\end{proof}


If a normal integrand $h$ is L-bounded, then $h(x)$ is quasi-integrable and $E^\G [h(x)]$ is well-defined for every $x\in L^\infty(\G)$. The following lemma shows that, for L-bounded normal integrands, it suffices to test with $x\in L^\infty(\G)$ in the definition of conditional expectation.

\begin{lemma}\thlabel{lem:citest}
Given an L-bounded normal integrand $h$, a $\G$-normal integrand $\bar h$ is the $\G$-conditional expectation of $h$ if and only if
\[
\bar h(x) = E^\G[h(x)] \quad a.s.
\]
for all $x\in L^\infty(\G)$. In particular, if $h(x)$ is $\G$-measurable for every $x\in L^\infty(\G)$, then $E^\G h=h$.
\end{lemma}

\begin{proof}
It is clear that the condition holds if $\bar h=E^\G h$. On the other hand, if $\bar h\ne E^\G h$, then there exists an $x\in L^0(\G)$ such that $h(x)$ is quasi-integrable and $\bar h(x)\ne E^\G [h(x)]$. For $\nu$ large enough, 
\[
\one_{\{|x|\le \nu\}}\bar h(x)\ne \one_{\{|x|\le \nu\}}E^\G [h(x)].
\]
The left side equals $\one_{\{|x|\le \nu\}}\bar h(\one_{\{|x|\le \nu\}}x)$ while the right side equals
\begin{align*}
  \one_{\{|x|\le \nu\}}E^\G [h(x)] &= \one_{\{|x|\le \nu\}}E^\G[\one_{\{|x|\le \nu\}}h(x)]\\
  &= \one_{\{|x|\le \nu\}}E^\G[\one_{\{|x|\le \nu\}}h(\one_{\{|x|\le \nu\}}x)]\\
&= \one_{\{|x|\le \nu\}}E^\G[h(\one_{\{|x|\le \nu\}}x)].
\end{align*}
Thus, $\bar h(\one_{\{|x|\le \nu\}}x)\ne E^\G[h(\one_{\{|x|\le \nu\}}x)]$ so the condition fails. By \cite[Lemma~6]{thi81}, the condition defines $\bar h$ uniquely almost surely everywhere.
\end{proof}

\begin{example}\thlabel{ex:celq}
Given $Q\in L^1(\reals^{n\times n})$, $v\in L^1(\reals^n)$, $m\in L^1$, the function
\[
h(x,\omega):=\frac{1}{2}x\cdot Q(\omega)x+x\cdot v(\omega)+m(\omega).
\]
is a normal integrand, by \cite[Example~14.29]{rw98}. If $Q$ is almost surely positive definite, then $h$ is L-bounded and
\[
(E^\G h)(x,\omega) = \frac{1}{2}x\cdot E^\G[Q](\omega)x+x\cdot (E^\G v)(\omega) + (E^\G m)(\omega).
\]
\end{example}

\begin{proof}
Given $x\in L^\infty$, \thref{lem:ce} gives
\[
E^G[h(x)]=\frac{1}{2}x\cdot E^\G[Q](\omega)x+x\cdot (E^\G v)(\omega)+(E^\G m)(\omega),
\]
so the claim follows from \thref{lem:citest}.
\end{proof}

\begin{theorem}\thlabel{thm:existenceCar}
Let $h$ be a real-valued normal integrand such that there exist $\bar x\in\reals^n$ and $\rho\in L^1$ with $h(\bar x)\in L^1$ and
\[
|h(x)-h(x')|\le\rho |x-x'| \quad\forall x,x'\in\reals^n.
\]
Then $E^\G h$ exists and it is characterized by
\[
(E^\G h)(x) = E^\G[h(x)] \quad \forall x\in\reals^n.
\]
Moreover, 
\begin{equation}\label{eq:lipschitz}
|(E^\G h)(x-x')|\le (E^\G\rho) |x-x'| \quad\forall x,x'\in\reals^n.
\end{equation}
\end{theorem}

\begin{proof}
The assumptions imply that $h(x)\in L^1$ for all $x\in\reals^n$ and that, by Jensen's inequality,
\[
|E^\G[h(x)]-E^\G[h(x')]|\le (E^\G\rho)|x-x'|\quad a.s.
\]
for all $x,x'\in\reals^n$. Let $D$ be a countable dense set in $\reals^n$ and define $\tilde h(x,\omega):=E^\G[h(x)](\omega)$ for each $x\in D$. By countability of $D$, there is a $P$-null set $N\in\G$ such that, for all $\omega\in\Omega\setminus N$,
\[
|\tilde h(x,\omega)-\tilde h(x',\omega)|\le (E^\G\rho)(\omega)|x-x'|\quad \forall x,x'\in D.
\]
The function $\tilde h$ has a unique continuous extension to $\reals^n\times(\Omega\setminus N)$. Finally, we extend the definition of $\tilde h$ to all of $\reals^n\times\Omega$ by setting $\tilde h(\cdot,\omega)=0$ for $\omega\in N$. The function $\tilde h$ thus constructed is a $\G$-measurable Caratheodory integrand and thus, normal. It is clear that it satisfies \eqref{eq:lipschitz} as well.

If $x=\sum_{i=1}^\nu x^i\one_{A^i}$, where $x^i\in D$ and $A^i\in\G$ form a disjoint partition of $\Omega$, we have
\[
E^\G[h(x)] = E^\G [h(\sum_{i=1}^\nu x^i\one_{A^i})] = E^\G[\sum_{i=1}^\nu \one_{A^i} h(x^i)] =  \sum_{i=1}^\nu \one_{A^i} \tilde h(x^i) =\tilde h(x). 
\]
Any $x\in L^\infty(\G)$ is an almost sure limit of such simple random variables bounded by $\|x\|_{L^\infty}$, so dominated convergence for conditional expectations (see e.g.\ \cite[Theorem~II.7.2]{shi96}) and scenariowise continuity of $h$ and $\tilde h$ imply that $E^\G[h(x)]=\tilde h(x)$. Thus Lemma~\ref{lem:citest} implies the claim.
\end{proof}

\begin{cexample}
The claim of \thref{thm:existenceCar} fails if $\rho$ is not integrable. Indeed, let $\xi$ be uniformly distributed on $(0,1)$, $\eta>1$ be nonintegrable independent of $\xi$ and $\G$ be generated by $\xi$. Let $h(x,\omega)=\eta(\omega)|x-\xi(\omega)|$. Then $E^\G [h(x)]=+\infty$ for every constant $x\in\reals$. On the other hand, choosing $x=\xi$, we have $E^\G[h(x)]=0$, so $E^\G h$ is not characterized by constants even though $h$ is Lipschitz and L-bounded convex normal integrand.
\end{cexample}

\begin{lemma}\thlabel{lem:cp1}
Let $h^1$ and $h^2$ be L-bounded normal integrands with $h^1\le h^2$. Then $E^\G h^1\le E^\G h^2$ whenever the conditional expectations exist.
\end{lemma}

\begin{proof}
For any $x\in L^\infty(\G)$, $h^1(x)$ and $h^2(x)$ are quasi-integrable, so $h^1\le h^2$ implies
\[
(E^\G h^1)(x)=E^\G[h_1(x)]\le E^\G[h_2(x)]=(E^\G h^2)(x)
\]
and the result follows from \cite[Lemma~6]{thi81}.
\end{proof}

The following result is a monotone convergence theorem for conditional expectations of integrands. 

\begin{theorem}\thlabel{thm:mon}
Let $(h^\nu)_{\nu=1}^\infty$ be a nondecreasing sequence of L-bounded normal integrands and 
\[
h=\sup_\nu h^\nu.
\]
If each $E^\G h^\nu$ exists, then $E^\G h$ exists and 
\[
E^\G h=\sup_\nu E^\G h^\nu.
\]
\end{theorem}

\begin{proof}
For any $x\in L^\infty(\G)$ and $\alpha\in L^\infty(\G;\reals_+)$, monotone convergence and Lemma~\ref{lem:cp1} imply that
\begin{align*}
E[\alpha{h}(x)] &= E[\alpha\sup_\nu h^\nu(x)]= \sup_\nu E[\alpha h^\nu(x)]\\
&= \sup_\nu E[\alpha E^\G[h^\nu(x)]]=  E[\alpha\sup_\nu E^\G[h^\nu(x)]].
\end{align*}
Thus $\sup_\nu E^\G h^\nu=E^\G h$.
\end{proof}

The following is our main result on the existence of conditional normal integrands.

\begin{theorem}\thlabel{thm:existce}
An L-bounded (convex) normal integrand admits a conditional expectation and that is L-bounded (and convex) as well. 
\end{theorem}

\begin{proof}
Let $h$ be an L-bounded normal integrand. Assume first that $h\le m$ for some constant $m>0$. By \cite[Example~9.11]{rw98},  
\[
h^\nu(x,\omega):=\inf_{x'} \{h(x',\omega)+\nu \rho(\omega)|x-x'|\}
\]
form a nondecreasing sequence of Caratheodory functions increasing pointwise to $h$. The assumed upper bound and the L-boundedness of $h$ imply that $h^\nu$ satisfy the assumptions of Theorem~\ref{thm:existenceCar}. 
Thus, by Theorems~\ref{thm:existenceCar} and \ref{thm:mon}, $E^\G h$ exists. To remove the assumption $h\le m$, consider the nondecreasing sequence of functions $h^n(x):=\min\{h(x),m\}$ (which are normal integrands, by \cite[Proposition~14.44]{rw98}) and apply Theorem~\ref{thm:mon} again. The preservation of convexity follows from \cite[Proposition~1.6.2]{tru91}.
\end{proof}

The truncation argument in the above proof is adapted from \cite{thi81}. It gives the existence under slightly more general conditions than \cite{bis73} who assumed the existence of a $\G$-measurable $x$ such that $h(x)$ is integrable. More general existence results in the nonconvex case have been given in \cite{tru91,chs3}.

\subsection{Conditional expectations in operations}\label{sec:properties}

Most results in this section can be found in \cite{bis73} \cite{hu77} and \cite{tru91} with the exception of \thref{thm:cini} and parts 2 and 3 of \thref{thm:condoperations} which seem new. \thref{thm:cerec} extends Corollary~1 of \cite[Theorem~3]{bis73} by slightly relaxing the assumptions on the domains of the integral functionals. 

Recall that if $\xi$ is a quasi-integrable random variable and $\G'$ is a sub-$\sigma$-algebra of $\G$, then
\[
E^{\G'}[E^{\G}\xi] = E^{\G'}\xi;
\]
see \thref{tp}. This extends to normal integrands as follows.
\begin{theorem}[Tower property]\thlabel{thm:tower}
Assume that $h$ is an L-bounded normal integrand, and $\G'$ is a sub-$\sigma$-algebra of $\G$. Then
\[
E^{\G'} (E^\G h)=E^{\G'} h.
\]
\end{theorem}
\begin{proof}
By Theorem~\ref{thm:existce}, $E^\G h$ exists and is L-bounded, so the result follows from the usual tower property (see \thref{tp}) and Lemma~\ref{lem:citest}.
\end{proof}

Given $\H\subset\F$, $\sigma$-algebras $\G$ and $\G'$ are {\em $\H$-conditionally independent} if
\[
E^\H[1_{A'} 1_A] = E^\H[1_{A'}] E^\H[1_A]
\]
for every $A\in \G$ and $A'\in \G'$.  A random variable $w$ is {\em $\H$-conditionally independent} of $\G$ if $\sigma(w)$ and $\G$ are $\H$-conditionally independent.  Likewise, we say that a normal integrand $h$ is {\em $\H$-conditionally independent} of $\G$ if $\sigma(h)$ and $\G$ are $\H$-conditionally independent. Here $\sigma(h)$ is the smallest $\sigma$-algebra under which $\epi h$ is measurable. In other words, $\sigma(h)$ is generated by the family
\[
\{(\epi h)^{-1}(O)\mid O\subset\reals^{n+1} \text{ open}\}.
\]

\begin{example}\thlabel{cindcomp}
Let $\xi$ be a random variable with values in a measurable space $(\Xi,\A)$, $H$ a $\A$-normal integrand on $\reals^n$ and
\[
h(x,\omega)=H(x,\xi(\omega)).
\]
 If $\xi$ is $\H$-conditionally independent of $\G$, then $h$ is so too. Indeed, given an open $O\subset\reals^{n+1}$, we have $(\epi h)^{-1}(O)=\xi^{-1}((\epi H)^{-1}(O))$, so $\sigma(h)\subseteq\sigma(\xi)$. To conclude, it suffices to note that sub-$\sigma$-algebras inherit conditional independence.
\end{example}

If an integrable random variable $w$ is $\H$-conditionally independent of $\G$, then
\[
E^{\G\vee\H} [w] = E^\H [w],
\]
by \thref{lem:condind}. This extends to normal integrands as follows.

\begin{theorem}\thlabel{thm:cini}
Let $h$ be an L-bounded normal integrand $\H$-conditionally independent of $\G$. Then 
\[
E^{\G\vee\H} h=E^\H h.
\]
In particular, if $h$ is independent of $\G$, then $E^\G h$ is deterministic. 
\end{theorem}

\begin{proof}
Assume first that $h$ satisfies the assumptions of Theorem~\ref{thm:existenceCar}. Then $E^\H h$ is characterized by
\[
E^\H (h(x)) = (E^\H h)(x)\quad\forall x\in \reals^n
\]
and likewise for $E^{\G\vee\H}h$. Thus $E^{\G\vee\H}h = E^\H h$, by Lemma~\ref{lem:condind}. The first claim now follows using Lipschitz regularizations as in the proof of Theorem~\ref{thm:existce}. The second claim follows by taking $\H$ the trivial $\sigma$-algebra.
\end{proof}

\begin{remark}
Conditional expectation is a linear operator on the linear space of normal integrands that satisfy the assumptions of \thref{thm:existenceCar}.
\end{remark}
\begin{proof}
If $h$ satisfies the assumptions of \thref{thm:existenceCar}, then so does $-h$, and $h(x)\in L^1$ for every $x\in\reals^n$. Thus
\[
E^\G[-h](x)=E^\G[-h(x)]=-E^\G[h(x)]=-(E^\G h)(x) \quad\forall x\in\reals^n,
\]
so $E^\G[-h]=-E^{\G} h$, by \thref{thm:existenceCar}. Additivity is proved similarly.
\end{proof}

The class of normal integrands is not a linear space, so one cannot hope for linearity of the conditional expectation, in general.

The product of two extended real numbers is defined as zero if one of them is zero while the nonnegative scalar multiple of a function $h$ is defined by
\[
(\alpha h)(x) :=
\begin{cases}
  \alpha h(x) & \text{if $\alpha>0$},\\
  \delta_{\cl\dom h}(x) & \text{if $\alpha=0$}.
\end{cases}
\]
Equivalently,
\[
(\alpha h)(x) := \alpha h(x)+\delta_{\cl\dom h}(x)\quad\alpha\ge 0.
\]

\begin{theorem}\thlabel{thm:condoperations}
Let $h$, $h^1$ and $h^2$ be L-bounded normal integrands.
\begin{enumerate}
\item\label{opersum}
  $h^1+h^2$ is L-bounded and $E^\G(h^1+h^2)=E^\G h^1+E^\G h^2$. 
\item
  If $\alpha \in L^1_+$, $h$ is $\G$-measurable and $\alpha h$ is L-bounded, then $E^\G(\alpha h)=E^\G[\alpha]h$.
\item\label{opermult}
  If $\alpha\in L^0_+(\G)$ and $\alpha h$ is L-bounded, then $E^\G(\alpha h) = \alpha E^\G h$ if either $\alpha$ is strictly positive or $F^\G[\cl\dom h]=\cl\dom E^\G h$.
\end{enumerate}
\end{theorem}

\begin{proof}
Let $x\in L^\infty(\G)$. Since $h^1(x)^-$ and $h^2(x)^-$ are integrable, \ref{opersum} follows from \thref{lem:citest} and the first part of \thref{lem:ce}. In 2,
\begin{align*}
  E^\G[(\alpha h)(x)] &= E^\G[\alpha h(x) + \delta_{\cl\dom h}(x)]\\
  &= E^\G[\alpha]h(x) + \delta_{\cl\dom h}(x),\\
  &=(E^\G[\alpha]h)(x),
\end{align*}
by \thref{lem:ce} so the claim follows from \thref{lem:citest} again. In 3,
\begin{align*}
  E^\G[(\alpha h)(x)] &= E^\G[\alpha h(x) + \delta_{\cl\dom h}(x)]\\
  &=\alpha E^\G[h(x)] + E^\G[\delta_{\cl\dom h}(x)],
\end{align*}
by \thref{lem:ce}. If $\alpha$ is strictly positive or if $F^\G[\cl\dom h]=\cl\dom E^\G h$, we thus get
\begin{align*}
  E^\G[(\alpha h)(x)] &=\alpha(E^\G h)(x) + \delta_{\cl\dom E^\G h}(x)\\
  &= (\alpha E^\G h)(x),
\end{align*}
which completes the proof.
\end{proof}

The following shows that part 3 of \thref{thm:condoperations} may fail without the extra assumptions.

\begin{cexample}
Let  $\G$ be trivial, $\alpha=0$ almost surely and $h(x,\omega):=\eta|x|$, where  $\eta\notin L^1$ is nonnegative. Then  $\alpha h=0$ while $E^\G h=\delta_{\{0\}}$, so
\[
E^\G(\alpha h) \ne \alpha E^\G h.
\]
\end{cexample}

\begin{theorem}\thlabel{thm:cerec}
Assume that $h$ is a convex normal integrand such that there exists $x\in \dom Eh\cap L^0(\G)$ and $v\in \dom Eh^*\cap L^1$ with $(x\cdot v)^-\in L^1$. Then
\[
E^\G(h^\infty)=(E^\G h)^\infty.
\]
\end{theorem}

\begin{proof}
The difference quotients
\[
h^\lambda(x',\omega):=\frac{h(x(\omega)+\lambda x',\omega)-h(x(\omega),\omega)}{\lambda}
\]
define a sequence of normal integrands $(h^\lambda)_{\lambda=1}^\infty$ that, by convexity, increase pointwise to $h^\infty$. By Fenchel's inequality,
\[
h^\lambda(x',\omega) \ge x' \cdot v(\omega)+x\cdot v(\omega)-h^*(v(\omega),\omega))-h(x(\omega),\omega),
\]
so the claim follows from \thref{thm:mon,thm:condoperations}.
\end{proof}







\section{Dynamic programming for lower bounded objectives}\label{sec:dplb}

The optimality conditions in \thref{thm:dp0} are essentially from \cite{pp12} but formulated here more generally. The existence results for the generalized Bellman equations for lower bounded integrands in \thref{thm:belb} are from \cite{pp12}. They extend those of \cite{rw76,evs76} in the convex case by relaxing compactness assumption on the set of feasible strategies. Lemma~6 of \cite{pp12} gives also the converse of \thref{thm:belb} in the sense that \thref{ass:dplb} necessarily holds if the generalized Bellman equations admit a solution and the sets $N_t$ are linear for all $t$.

This section studies dynamic programming in the case where $h$ is {\em lower bounded} in the sense that there exists an $m\in L^1$ such that
\[
h(x,\omega)\ge m(\omega)\quad\forall x\in\reals^n
\]
almost surely. 

\begin{assumption}\thlabel{ass:dplb}
Problem \eqref{sp} is feasible, $h$ is lower bounded and
\[
\L:=\{x\in\N \mid h^\infty(x)\le 0\}
\]
is a linear space.
\end{assumption}

The following is from \cite[Lemma~4]{pp12}.

\begin{theorem}\thlabel{thm:belb}
Under \thref{ass:dplb}, \eqref{be} has a unique solution $(h_t)$ of lower bounded normal integrands and
\[
N_t:=\{x_t\in\reals^{n_t} \mid h_t^\infty(x^t,\omega)\le 0,\ x^{t-1}=0\}
\]
are linear-valued for all $t$. In this case, if $x\in\L$ is such that $x^{t-1}=0$ then $x_t\in N_t$ almost surely.
\end{theorem}

\thref{thm:dp0} below, shows that, if the Bellman equations \eqref{be} admit a solution $(h_t)_{t=0}^T$, then the optimal solutions $\bar x\in\N$ of \eqref{sp} are characterized by scenariowise minimization of $h_t$.

We will denote the projection of the set $\N$ of adapted strategies to its first $t$ components by
\begin{align*}
  \N^t :&= \{x^t\mid x\in\N\} = \left\{(x_{t'})_{t'=0}^t\midb x_{t'}\in L^0(\Omega,\F_{t'},P;\reals^{n_{t'}})\right\}.
\end{align*}
The lower boundedness condition in the following result will be relaxed in Section~\ref{sec:Lbexist} below.

\begin{theorem}\thlabel{thm:dp0}
Assume that $h$ is lower bounded, \eqref{sp} is feasible and that the Bellman equations \eqref{be} admit a solution $(h_t)_{t=0}^T$. Then each $h_t$ is lower bounded, 
\begin{align*}
\inf\eqref{sp} &=\inf_{x^t\in\N^t} Eh_t(x^t)
\end{align*}
for all $t=0,\dots,T$ and, moreover, an $\bar x\in\N$ solves \eqref{sp} if and only if
\begin{equation}\tag{OP}\label{oc0}
\bar x_t\in\argmin_{x_t\in\reals^{n_t}}h_t(\bar x^{t-1},x_t)\quad \text{a.s.}
\end{equation}
for all $t=0,\ldots,T$. If
\[
N_t(\omega):=\{x_t\in\reals^{n_t} \mid h_t^\infty(x^t,\omega)\le 0,\ x^{t-1}=0\}
\]
is linear-valued for all $t=0,\ldots,T$, then there exists an optimal $x\in\N$ with $x_t\perp N_t$ almost surely.
\end{theorem}

\begin{proof}
Let $x\in\N$. By \cite[Theorem~14.60]{rw98},
\begin{align*}
E \tilde h_{t-1}(x^{t-1}) &= \inf_{x_{t}\in L^0(\F_{t})}E h_{t}(x^{t-1},x_{t}).
\end{align*}
Since $\tilde h_{t-1}$ is bounded from below, $E\tilde h_{t-1} =Eh_{t-1}$ on $\N^{t-1}$ and thus
\[
\inf_{x^{t-1}\in\N^{t-1}} Eh_{t-1}(x^{t-1}) = \inf_{x^{t}\in\N^{t}} Eh_{t}(x^{t}).
\]
By induction, $\inf\eqref{sp} =\inf_{x^t\in\N^t} Eh_t(x^t)$.

To prove the second claim, note first that
\[
\bar x^{t}\in \argmin_{x^{t}\in\N^{t}} Eh_{t}(x^{t})
\]
 if and only if
\[
\bar x^{t-1}\in\argmin_{x^{t-1}\in\N^{t-1}} Eh_{t-1}(x^{t-1})\quad\text{and}\quad \bar x_{t} \in\argmin_{x_t \in L^0(\F_{t})} Eh_{t}(\bar x^{t-1},x_t),
\]
where, by the second part of \cite[Theorem~14.60]{rw98}, the second inclusion means that
\[
\bar x_t\in\argmin_{x_t\in\reals^{n_t}}h_t(\bar x^{t-1},x_t)\quad \text{a.s.}
\]
An $\bar x\in\N$ solves \eqref{sp} if and only if $\bar x$ minimizes $Eh_T$. A backward recursion shows that optimal solutions satisfy \eqref{oc0}. The converse follows from a forward recursion.

Applying \thref{thm:ip} recursively forward in time shows that \eqref{oc0} has an $\F_t$-measurable solution $\bar x_t\perp N_t$ almost surely for all $t=0,\ldots,T$. The last claim thus follows from the second one.
\end{proof}

Theorem~\ref{thm:dp0} can be thought of as a discrete-time version of a ``verification theorem'' which, in the context of continuous-time stochastic control gives conditions under which solutions of the Hamilton-Jacobi-Bellman equation characterize the optimum values and solutions of optimal control problems.

\section{Dynamic programming for L-bounded objectives}\label{sec:Lbexist}
 
This section extends the results of the previous section by relaxing the lower boundedness assumption on $h$. The extensions are largely based on \thref{lem:perp} which first appeared in \cite{per16} where it was used to extend the main result of \cite{pp12} on the lower semicontinuity of the optimum value function of \eqref{sp}. The extension is based on the interplay of the space $\N$ of adapted strategies with the set
\[
\N^\perp:=\{v\in L^1\mid E[x\cdot v]=0\ \forall x\in\N^\infty\},
\]
where $\N^\infty:=\N\cap L^\infty$. The following gives an alternative expression for $\N^\perp$.

\begin{lemma}\thlabel{perpchar}
$\N^\perp = \{v\in L^1\mid E_t[v_t]=0\quad t=0,\ldots,T\}$.
\end{lemma}

\begin{proof}
We have
\[
E[x\cdot v] = \sum_{t=0}^TE[x_t\cdot v_t]
\]
so $v\in\N^\perp$ if and only if $E[x_t\cdot v_t]=0$ for all $x_t\in L^\infty(\F_t)$. The claim now follows from \thref{lem:ce}.
\end{proof}

More interestingly, we have the following.

\begin{lemma}\thlabel{lem:perp}
Let $x\in\N$ and $v\in\N^\perp$. If $E[x\cdot v]^+\in L^1$, then $E[x\cdot v]=0$.
\end{lemma}

\begin{proof}
Assume first that $T=0$. Defining $x^\nu:=\one_{\{|x|\le \nu\}}x$, we have $x^\nu\in \N^\infty$, so $E[x^\nu\cdot v]=0$ and thus, $E[x^\nu\cdot v]^-=E[x^\nu\cdot v]^+$. Since $[x^\nu\cdot v]\le[x\cdot v]^+\in L^1$, Fatou's lemma gives
\[
E[x\cdot v]^- \le \liminf_{\nu\to\infty} E[x^\nu\cdot v]^- = \liminf_{\nu\to\infty} E[x^\nu\cdot v]^+ \le E[x\cdot v]^+
\]
so $[x\cdot v]^-\in L^1$ as well. Since $|x^\nu\cdot v|\le |x\cdot v|$, dominated convergence theorem gives $E[x\cdot v] = \lim E[x^\nu\cdot v]=0$.

Assume now that the claim holds for every $(T-1)$-period model. Defining $x^\nu :=\one_{\{|x_0|\le \nu\}} x$, we have
\[
[\sum_{t=1}^T x^\nu_t\cdot v_t]^+ = [x^\nu\cdot v-x^\nu_0\cdot v_0]^+ \le [x^\nu\cdot v]^+ +[x^\nu_0\cdot v_0]^- \le [x\cdot v]^+ +[x^\nu_0\cdot v_0]^-,
\]
where the right side is integrable. Thus, $E[\sum_{t=1}^T x^\nu_t\cdot v_t]=0$, by the induction hypothesis. Since $x^\nu_0\in L^\infty$, we also have $E[x_0^\nu\cdot v_0]=0$ so $E[x^\nu\cdot v]=0$. This implies $E[x\cdot v]=0$ just like in the case $T=0$.
\end{proof}

The following illustrates \thref{lem:perp} with the stochastic integral of an adapted process with respect to a martingale; see \cite{js98}. 

\begin{example}\thlabel{ex:siperp}
Assume that $n_t=d$ for all $t$ and let $s$ be a $d$-dimensional martingale, i.e.\ an adapted integrable stochastic process such that $E_t[\Delta s_{t+1}]=0$ for all $t$. If $x\in\N$ is such that
\[
E[\sum_{t=0}^{T-1} x_t\cdot\Delta s_{t+1}]^+<\infty,
\]
then $E[\sum_{t=0}^{T-1} x_t\cdot\Delta s_{t+1}]=0$. This follows from Lemma~\ref{lem:perp} with $v\in\N^\perp$ defined by $v_t=\Delta s_{t+1}$.
\end{example}

The extensions of \thref{thm:belb} and \thref{thm:dp0} below are based on the following lemma that allows us to reduce a more general problem to one with a lower bounded integrand.

\begin{lemma}\thlabel{lem:hkbe}
Assume that there exists $p\in\N^\perp$ such that the normal integrand $k(x,\omega):=h(x,\omega)-x\cdot p(\omega)$ is lower bounded. Then  \eqref{be} has a solution for $h$ if and only if \eqref{be} has a solution for $k$. In this case, the solutions are unique and related by
\begin{equation}\label{eq:lbt}
k_t(x^t,\omega)=h_t(x^t,\omega)-x^t\cdot E_t p_t(\omega).
\end{equation}
\end{lemma}

\begin{proof}
If \eqref{be} admits a solution for $h$, there exist sequences $(\tilde h_t)_{t=0}^{T-1}$ and $(h_t)_{t=0}^T$ of normal integrands that satisfy \eqref{be}. Since $k$ is lower bounded, there exists $m\in L^1$ such that
\[
h(x,\omega)\ge x\cdot p(\omega)-m(\omega).
\]
We claim that,
\begin{align}
  h_t(x^t,\omega) &\ge x^t\cdot [E_t p^t](\omega)- [E_t m](\omega)\label{eq:hlb}\\
  \intertext{for $t=T,\ldots,0$ and}
  \tilde h_t(x^t,\omega) &\ge x^t\cdot [E_{t+1} p^t](\omega)- [E_{t+1} m](\omega)\label{eq:tildehlb}
\end{align}
for $t=T-1,\ldots,0$. Indeed, if \eqref{eq:hlb} holds, then, since $E_tp_t=0$, we get \eqref{eq:tildehlb} for $t-1$. If \eqref{eq:tildehlb} holds, then \eqref{eq:hlb} follows from \thref{lem:cp1,ex:celq}. Since $h_T=E_Th$, \eqref{eq:hlb} holds for $T$ so the claim follows by induction on $t$.

The lower bounds in \eqref{eq:tildehlb} and \eqref{eq:hlb} imply that the sequences $(\tilde k_t)_{t=0}^{T-1}$ and $(k_t)_{t=0}^T$ defined by
\begin{align*}
  \tilde k_t(x^t,\omega) &:= \tilde h_t(x^t,\omega)-x^t\cdot [E_{t+1} p^t](\omega),\\
  k_t(x^t,\omega) &:= h_t(x^t,\omega)-x^t\cdot [E_t p^t](\omega)
\end{align*}
satisfy \eqref{be} for $k$. Indeed, by \thref{thm:condoperations,ex:celq}, $E_t\tilde k_t=k_t$ while, since $E_t p_t=0$,
\begin{align*}
  \tilde k_{t-1}(x^{t-1},\omega) &= \tilde h_{t-1}(x^{t-1},\omega)-x^{t-1}\cdot [E_t p^{t-1}](\omega),\\
  &=\inf_{x_t\in\reals^{n_t}}\{h_t(x^t,\omega)-x^t\cdot [E_t p^t](\omega)\}\\
  &= \inf_{x_t\in\reals^{n_t}}k_t(x^{t-1},x_t,\omega).
\end{align*}
Thus, $(k_t)_{t=1}^T$ solves \eqref{be} for $k$. Conversely, assume that $(k_t)_{t=0}^T$ and $(k_t)_{t=0}^{T-1}$ solve \eqref{be} for $k$. By \eqref{thm:belb}, $k_t$ are lower bounded. Similar argument as above, then shows that the functions
\begin{align*}
  \tilde h_t(x^t,\omega) &:= \tilde k_t(x^t,\omega)+x^t\cdot [E_{t+1} p^t](\omega),\\
  h_t(x^t,\omega) &:= k_t(x^t,\omega)+x^t\cdot [E_t p^t](\omega)
\end{align*}
satisfy \eqref{be} for $h$. This completes the proof.
\end{proof}

The following extends the existence result in \thref{thm:belb} by relaxing the lower boundedness assumption on $h$.

\begin{theorem}\thlabel{thm:be}
Assume that there exists $p\in\N^\perp$ such that the normal integrand $k(x,\omega):=h(x,\omega)-x\cdot p(\omega)$ is lower bounded and
\[
\{x\in\N\mid k^\infty(x)\le 0 \text{ a.s.}\}
\]
is a linear space. Then \eqref{be} has a unique solution $(h_t)_{t=0}^T$ of L-bounded normal integrands and
\[
N_t(\omega):=\{x_t\in\reals^{n_t} \mid h_t^\infty(x^t,\omega)\le 0,\ x^{t-1}=0\}
\]
are linear-valued for all $t$. 
\end{theorem}

\begin{proof}
By \thref{thm:belb}, the Bellman equations associated with $k$ have a unique solution $(k_t)_{t=0}^T$ of lower bounded normal integrands and the measurable mappings
\[
\hat N_t(\omega):=\{x_t\in\reals^{n_t} \mid k_t^\infty(x^t,\omega)\le 0,\ x^{t-1}=0\}
\]
are linear-valued for all $t$. Thus, by \thref{lem:hkbe}, \eqref{be} associated with $h$ has a unique solution of L-bounded normal integrands and, since $E_tp_t=0$, we have $N_t=\hat N_t$ almost surely.
\end{proof}

The following generalizes \thref{thm:dp0} by relaxing the lower boundedness assumption on $h$.

\begin{theorem}\thlabel{thm:dpLb}
Assume that there exists $p\in\N^\perp$ such that the normal integrand $k(x,\omega):=h(x,\omega)-x\cdot p(\omega)$ is lower bounded and $Ek(x)=Eh(x)$ for all $x\in\N$. If \eqref{sp} is feasible and the Bellman equations \eqref{be} admit a solution $(h_t)_{t=0}^T$, then
\begin{align*}
\inf\eqref{sp} &=\inf_{x^t\in\N^t} Eh_t(x^t)
\end{align*}
for all $t=0,\dots,T$ and, moreover, an $\bar x\in\N$ solves \eqref{sp} if and only if
\begin{equation}\tag{OP}\label{oc0}
\bar x_t\in\argmin_{x_t\in\reals^{n_t}}h_t(\bar x^{t-1},x_t)\quad \text{a.s.}
\end{equation}
for all $t=0,\ldots,T$. If
\[
N_t(\omega):=\{x_t\in\reals^{n_t} \mid h_t^\infty(x^t,\omega)\le 0,\ x^{t-1}=0\}
\]
is linear-valued for all $t=0,\ldots,T$, then there exists an optimal $x\in\N$ with $x_t\perp N_t$ almost surely.
\end{theorem}

\begin{proof}
By \thref{lem:hkbe}, \eqref{be} has a solution for $k$. Since $k$ is lower bounded and since $Eh=Ek$ on $\N$, \thref{thm:dp0} says that
\begin{equation}\label{eq:k_t}
\inf\eqref{sp}=\inf_{x^t\in\N^t} Ek_t(x^t)
\end{equation}
for all $t=0,\ldots,T$ and that an $\bar x\in\N$ solves \eqref{sp} if and only if
\begin{equation*}
\bar x_t\in\argmin_{x_t\in\reals^{n_t}}k_t(\bar x^{t-1},x_t)\quad \text{a.s.}
\end{equation*}
for all $t=0,\ldots,T$.

By definition of $(h_t)_{t=0}^T$, we always have $\inf\eqref{sp}\ge\inf_{x^t\in\N^t} Eh_t(x^t)$. By \eqref{eq:lbt}, there exist $m_t\in L^1$ such that
\[
h_t(x^t) \ge x^t\cdot E_t p^t - m_t.
\]
If $Eh_t(x^t)<\infty$, then $E[x^t\cdot(E_t p^t)]<\infty$, so Lemma~\ref{lem:perp} gives $Ek_t(x^t)=Eh_t(x^t)$. Thus, 
\[
\inf\eqref{sp}=\inf_{x^t\in\N^t} Eh_t(x^t),
\]
which proves the first claim. Since $E_tp_t=0$, we have
\[
\argmin_{x_t}k_t(x^{t-1}(\omega),x_t,\omega) = \argmin_{x_t}h_t(x^{t-1}(\omega),x_t,\omega),
\]
which proves the second claim. The last claim follows from that of \thref{thm:dp0} since $E_tp_t=0$.
\end{proof}

We combine the assumptions of \thref{thm:belb,thm:dpLb} into the following.

\begin{assumption}\thlabel{ass:dp}
Problem \eqref{sp} is feasible and there exists $p\in\N^\perp$ such that
\begin{enumerate}
\item $k(x,\omega):=h(x,\omega)-x\cdot p(\omega)$ is lower bounded,
\item $Ek(x)=Eh(x)$ for all $x\in\N$,
\item $\{x\in\N\mid k^\infty(x)\le 0 \text{ a.s.}\}$ is a linear space.
\end{enumerate}
\end{assumption}

Note that, if $h$ is lower bounded, one can take $p=0$ so \thref{ass:dp} reduces to \thref{ass:dplb}. Sufficient conditions are given in \thref{2lambda,domsc} below. Applications in Section~\ref{sec:dpapp} illustrate these conditions further. In particular, in financial mathematics, the assumption is related to the existence of a martingale measure of the price process.  

The following combines \thref{thm:belb,thm:dpLb}.

\begin{theorem}\thlabel{thm:dplbc}
Under \thref{ass:dp}, \eqref{be} has a unique solution $(h_t)_{t=0}^T$,
\[
\inf\eqref{sp}=\inf_{x^t\in\N^t} Eh_t(x^t)
\]
for all $t=0,\ldots,T$, \eqref{sp} has a solution and the solutions $\bar x\in\N$ of  \eqref{sp} are characterized by
\begin{equation}\tag{OP}\label{op}
\bar x_t\in\argmin_{x_t\in\reals^{n_t}}h_t(\bar x^{t-1},x_t)\quad \text{a.s.}\quad t=0,\ldots,T.
\end{equation}
\end{theorem}

The following gives sufficient conditions for \thref{ass:dp}.

\begin{lemma}\thlabel{2lambda}
\thref{ass:dp} holds if \eqref{sp} is feasible and
\begin{enumerate}
\item
  there exists a $p\in\N^\perp$ and $m\in L^1$ and $\epsilon>0$ with 
  \[
  h(x,\omega)\ge \lambda x\cdot p(\omega)-m(\omega)\quad \forall \lambda\in[1-\epsilon,1+\epsilon],
  \]
\item
  $\{x\in \N \mid h^\infty(x)\le 0 \text{ a.s.}\}$ is a linear space.
\end{enumerate}
In this case
\[
 \{x\in \N \mid h^\infty(x)\le 0 \text{ a.s.}\} = \{x\in \N \mid k^\infty(x)\le 0 \text{ a.s.}\}.
\]
The lower bound in 1 can be written equivalently as 
\[
h(x,\omega)\ge x\cdot p(\omega)+ \epsilon |x\cdot p(\omega)| -m(\omega)
\]
or as
\[
\lambda p \in \dom Eh^*\quad \forall \lambda\in[1-\epsilon,1+\epsilon].
\]
\end{lemma}

\begin{proof}
We have
\begin{align*}
h(x,\omega)&\ge x\cdot p(\omega)-m(\omega),\\
h(x,\omega) - x\cdot p(\omega) &\ge \epsilon x\cdot p(\omega)-m(\omega).
\end{align*}
Let $x\in\N$. If either $Eh(x)<\infty$ or  $E[h(x)-x\cdot p]<\infty$, the above inequalities and Lemma~\ref{lem:perp}  give $E[x\cdot p]=0$, so 
\[
Eh(x) = E[h(x)-x\cdot p].
\]
The above inequalities also give
\begin{align*}
h^\infty(x,\omega)&\ge x\cdot p(\omega),\\
h^\infty(x,\omega) - x\cdot p(\omega) &\ge \epsilon x\cdot p(\omega).
\end{align*}
If either $h^\infty(x,0)\le 0$ or  $h^\infty(x,0)-x\cdot p \le 0$ almost surely, then $x\cdot p\le 0$ almost surely. Lemma~\ref{lem:perp} then implies $x\cdot p=0$ almost surely, so
\[
\{x\in \N \mid h^\infty(x)\le 0 \text{ a.s.}\} = \{x\in \N \mid h^\infty(x)- x\cdot p\le 0 \text{ a.s.}\}.
\]
The given linearity condition thus implies that in \thref{ass:dp}.
\end{proof}


\begin{lemma}\thlabel{domsc}
\thref{ass:dp} holds if \eqref{sp} is feasible and there exists $p\in\N^\perp$ such that
\begin{enumerate}
\item
  $k(x,\omega):=h(x,\omega)-x\cdot p(\omega)$ is lower bounded,
\item $\{x\in\N\mid x\in\dom h\ a.s.\}\subset\dom Eh$,
\item $\{x\in\N\mid k^\infty(x)\le 0\ a.s.\}$ is a linear space.
\end{enumerate}
\end{lemma}

\begin{proof}
Let $x\in\N$ be such that either $Ek(x)$ or $Eh(x)$ is finite. Then $x\in\dom h$ almost surely, so $x\in\dom Eh$ by 2. By 1, there exists $m\in L^1$ such that $h(x) =k(x)+x\cdot p \ge x\cdot p -m$, so $E[x\cdot p]=0$ by \thref{lem:perp}. Thus $Ek(x)=Eh(x)$.
\end{proof}

\begin{cexample}\label{ex:nops}
Without \thref{ass:dp}, it is possible that \eqref{sp} does not have a solution albeit \eqref{be} has a unique solution and there is a unique $x$ satisfying \eqref{op}.
 
Indeed, let $n_t=1$, $\alpha\in L^2(\F_0)$ and $p\in\N^\perp$ be such that $E [\alpha p_0]=\infty$ and consider
\[
h(x,\omega):= \frac{1}{2}|x_0-\alpha(\omega)|^2 +x_0 p_0(\omega).
\]
Here $h_0(x,\omega)=  \frac{1}{2}|x_0-\alpha(\omega)|^2$, so $x\in \N$ satisfies \eqref{op} if and only if $x_0=\alpha$. For such $x$, $Eh(x)=\infty$.
\end{cexample}

\begin{cexample}\label{ex:nops2}
It is possible that \eqref{be} has a solution, 
\[
\L:=\{x\in\N\mid h^\infty (x)\le 0\}
\]
is a linear space but $N_t$ in \thref{thm:dpLb} is not linear-valued.

Indeed, let $T=0$, $n_0=2$, $\F_0$ trivial, $p\in\N^\perp$ with $p$ nonzero almost surely and consider
\[
h(x,\omega):= (x^2_0-x^1_0)^+ +\delta_{\reals_+}(x^2_0) +x^1_0 p^1_0(\omega).
\]
Here $h_0(x,\omega)=  (x^2_0-x^1_0)^++\delta_{\reals_+}(x^2_0)$, so $N_0=\{x_0\mid 0\le x^2_0\le x^1_0\}$ while $\L=\{x \mid 0\le x^2_0\le x^1_0, x^1_0\cdot p^1_0\le 0 \text{ a.s.}\}=\{0\}$.
\end{cexample}

\section{Applications}\label{sec:dpapp}

This section applies \thref{thm:dp0,thm:dpLb} to some well-known instances of \eqref{sp}. Many of the existence results and optimality conditions below have previously been known only under more restrictive compactness and boundedness conditions. In the case of portfolio optimization in Section~\ref{sec:fm2} below, we extend earlier results by allowing for portfolio constraints.

\subsection{Mathematical programming}\label{sec:mp2}

Consider the problem
\begin{equation}\label{mp}\tag{$MP$}
\begin{aligned}
&\minimize\quad & Ef_0(x)&\quad\ovr\ x\in\N,\\
  &\st\quad & f_j(x) &\le 0\quad j=1,\ldots,l\ a.s.,\\
   & & f_j(x) &= 0\quad j=l+1,\ldots,m\ a.s.,
\end{aligned}
\end{equation}
where $f_j$, $j=0,\ldots,m$ are convex normal integrands with $f_j$ affine for $j>l$. The problem fits the general framework with
 \[
h(x,\omega):= \begin{cases}
f_0(x,\omega)\quad &\text{if } f_j(x,\omega)\le 0, j=1\dots,l, f_j(x,\omega)=0, j=l+1,\dots,m\\
+\infty\quad&\text{otherwise}.
\end{cases}
\]
Indeed, by \cite[Example~1.32 and Proposition~14.33]{rw98}, $h$ is a normal integrand. Problem \eqref{mp} is essentially from \cite{rw78} where it was analyzed through convex duality. We have extended the formulation by the inclusion of affine equality constraints.

\begin{assumption}\thlabel{ass:mp2}
\mbox{}
\begin{enumerate}
\item \eqref{mp} is feasible,
\item $\{x\in\N\mid f^\infty_j(x)\le 0\ j=0,\dots l, f^\infty_j(x)=0\ j=l+1,\dots,m\}$ is a linear space,
\item  there exists a $p\in\N^\perp$, $m\in L^1$ and an $\epsilon>0$ such that
\[
f_0(x)\ge x\cdot p + \epsilon|x\cdot p| - m \quad \text{a.s.}
\]
for all $x\in\reals^n$ with $f_j(x)\le 0$ for all $j=1,\dots, l$ and $f_j(x)=0$ for all $j=l+1,\dots, m$  almost surely.
\end{enumerate}
\end{assumption}
\thref{ass:mp2} relaxes the conditions imposed in \cite{rw76}. The last two conditions are sufficient for the conditions of \thref{2lambda} which imply \thref{ass:dp}. The last condition in \thref{ass:mp2} holds, in particular, if $f_0$ is bounded from below by an integrable function as one can then take $p=0$. A direct application of  \thref{thm:dplbc} gives the following.

\begin{theorem}\thlabel{thm:mp4}
Under \thref{ass:mp2}, $\eqref{mp}$ has a solution.
\end{theorem}

\begin{example}[Linear programming]\thlabel{ex:lp}
In the special of Example~\ref{ex:lp}, \thref{ass:mp2} means that
\begin{enumerate}
\item \eqref{mp} is feasible,
\item $\{x\in\N\mid c\cdot x\le 0, Ax\in K\}$ is a linear space,
\item  there exists a $p\in\N^\perp$, $m\in L^1$ and an $\epsilon>0$ such that
\[
c\cdot x\ge x\cdot p + \epsilon|x\cdot p| - m \quad \text{a.s.}
\]
for all $x\in\reals^n$ with $Ax-b\in K$  almost surely.
\end{enumerate}
\end{example}

The above are merely examples how the results of Section~\ref{sec:Lbexist} can be used. In some applications, the conditions of \thref{domsc} could be more convenient. The above only gives the existence of solutions. When the constraint matrix in \thref{ex:lp} has a block-diagonal form, one can write the dynamic programming recursion in a more familiar form; see \thref{bdslp} below.

\subsection{Optimal stopping}\label{ssec:os2}

Let $R$ be a real-valued adapted stochastic process and consider the {\em optimal stopping problem}
\begin{equation}\label{os}\tag{$OS$}
  \maximize\quad ER_\tau\quad\ovr \tau\in\T,
\end{equation}
where $\T$ is the set of {\em stopping times}, i.e.\ measurable functions $\tau:\Omega\to\{0,\ldots,T+1\}$ such that $\{\omega\in\Omega\mid \tau(\omega)\le t\}\in\F_t$ for each $t=0,\ldots,T$. Choosing $\tau=T+1$ is interpreted as not stopping at all. Accordingly, we define $R_{T+1}:=0$. Consider also the problem
\begin{equation}\label{ros}\tag{$ROS$}
\maximize_{x\in\N_+}\quad E \sum_{t=0}^TR_t x_t\quad\st\quad x\ge 0,\ \sum_{t=0}^Tx_t\le 1\ a.s.
\end{equation}
This is a convex relaxation of \eqref{os} which is obtained from \eqref{ros} by adding the constraint that $x_t\in\{0,1\}$. Indeed, the feasible strategies $x$ are then in one-to-one correspondence with stopping times via 
\[
x_t=\begin{cases}
1 &\text{if $t=\tau$},\\
0 &\text{if $t\ne\tau$}.
\end{cases}
\]
The following motivates the relaxation.

\begin{lemma}
If $R\in L^1$, then $\sup\eqref{os}=\sup\eqref{ros}$ and $\argmax\eqref{ros}$ is the closed convex hull of strategies $x$ that correspond to optimal solutions of \eqref{os}.
\end{lemma}

\begin{proof}
As in the proof of \cite[Lemma~2]{pp220b}, it can be shown that the processes corresponding to stopping times are the extreme points of the feasible set of \eqref{ros}.   By Banach-Alaoglu theorem, the feasible set of \eqref{ros} is compact in the weak topology that $L^\infty$ has as the dual of $L^1$. Thus, by Krein--Milman theorem, the feasible set of \eqref{ros} is the closed convex hull of those $x$ corresponding to stopping times.  

If $R\in L^1$, the objective is weakly continuous so the relaxation does not affect the optimum value. It is easy to verify, by contradiction, that extreme points of the weakly compact $\argmax\eqref{ros}$ are extreme points of the feasible set, which, by the Krein-Milman theorem again, proves the last claim. 
\end{proof}
Problem \eqref{ros} fits the general framework with $n_t=1$ for all $t$ and
\[
h(x,\omega) = 
\begin{cases}
-\sum_{t=0}^TR_t(\omega) x_t & \text{if $x\ge 0$ and $\sum_{t=0}^Tx_t\le 1$},\\
+\infty & \text{otherwise},
\end{cases}
\]
for an adapted real-valued process $R$ and $x_{-1}:=0$.

Let $S$ be the {\em Snell envelope} of $R$, i.e.\ the adapted stochastic process given by
\begin{align*}
S_{T+1}&:=0\\
S_t&:=\max\{R_t,E_t S_{t+1}\}.
\end{align*}
The Snell envelope is the smallest supermartingale that dominates the positive part $R^+$ of the reward process $R$. Indeed, let $\tilde S$ be another supermartingale that dominates $R^+$. Then $\tilde S_T\ge S_T$ and
\[
\tilde S_t\ge\max\{R_t,E_t\tilde S_{t+1}\}
\]
so $\tilde S_t\ge S_t$ for all $t$, by induction.

\begin{theorem}
Assume that $R\in L^1$. The optimum value of \eqref{ros} coincides for all $t=0,\ldots,T$ with that of
\begin{equation*}
\begin{aligned}
  &\maximize_{x^t\in\N^t}\quad & & E\left[\sum_{s=0}^t R_s x_s + E_t[S_{t+1}](1-\sum_{s=1}^t x_s)\right]\\
  &\st\quad & & x^t\ge 0,\ \sum_{s=0}^t x_s\le 1\quad \text{a.s.}
\end{aligned}
\end{equation*}
In particular, the optimum value of \eqref{ros} is $ES_0$. An $x\in\N$ is optimal if and only if
\[
x_t\in\argmax_{x_t\in\reals}\left\{(R_t-E_t[S_{t+1}])x_t\midb x_t\in[0,1-\sum_{s=0}^{t-1}x_s]\right\}\quad \text{a.s.}\quad t=0,\ldots,T.
\]
In particular, optimal values of \eqref{os} and \eqref{ros} coincide, \eqref{os} admits optimal solutions $\tau\in\T$ and they are characterized by  $R_\tau=S_\tau$ almost surely.
\end{theorem}

\begin{proof}

The domain of $h$ is contained in the unit simplex almost surely, so $h^\infty=\delta_{\{0\}}$ and, since $R\in L^1$, $h$ has an integrable lower bound. Thus, by \thref{thm:be}, BE has a solution $(h_t)_{t=0}^T$.  To prove the claims concerning \eqref{ros}, it suffices, by \thref{thm:dp0}, to show that
\[
h_t(x^t,\omega)=\sum_{s=0}^t\left[-R_s(\omega)x_s+\delta_{\reals_+}(x_s)\right]-E_t[S_{t+1}](\omega)(1-\sum_{s=0}^t x_s)+\delta_{\reals_+}(1-\sum_{s=0}^t x_s).
\]
Since $R$ is adapted, $h$ is $\F_T$-measurable so $h_T=h$ and the claim holds for $t=T$ since $S_{T+1}=0$. Assume that the claim holds for $t$. We get 
\[
\begin{split}
&\tilde h_{t-1}(x_{t-1},\omega)\\
 &:= \inf_{x_t\in\reals} h_t(x^t,\omega)\\
&= \inf_{x_t\in\reals}\left\{\sum_{s=0}^t\left[-R_s(\omega)x_s+\delta_{\reals_+}(x_s)\right] - E_t[S_{t+1}](\omega)(1-\sum_{s=0}^t x_s)\midb 0\le x_t\le 1-\sum_{s=0}^{t-1}x_s\right\}\\
&= \sum_{s=0}^{t-1}\left[-R_s(\omega)x_s+\delta_{\reals_+}(x_s)\right] - \sup_{x_t\in\reals}\left\{R_t(\omega)x_t+E_t[S_{t+1}](\omega)(1-\sum_{s=0}^t x_s)\midb 0\le x_t\le 1-\sum_{s=0}^{t-1}x_s\right\}\\
&= \sum_{s=0}^{t-1}[-R_s(\omega)x_s+\delta_{\reals_+}(x_s)] - \max\{R_t(\omega),E_t[S_{t+1}](\omega)\}(1-\sum_{s=0}^{t-1} x_s) + \delta_{\reals_+}(1-\sum_{s=0}^{t-1} x_s)\\
&= \sum_{s=0}^{t-1}[-R_s(\omega)x_s+\delta_{\reals_+}(x_s)] - S_t(\omega)(1-\sum_{s=0}^{t-1}x_s) + \delta_{\reals_+}(1-\sum_{s=0}^{t-1} x_s).
\end{split}
\]
By Example~\ref{ex:celq}, the $\F_{t-1}$-conditional expectation of the second last term is $E_{t-1}[S_t](\omega)(\sum_{s=0}^{t-1} x_s-1)$, so the claim holds for $t-1$. It is clear that the argmax over $x_t$ always contains either $0$ or $1-\sum_{s=0}^{t-1}x_s$. Thus, if $\sum_{s=0}^{t-1}x_s$ takes values in $\{0,1\}$, we can choose an optimal $x_t$ such that $\sum_{s=0}^{t}x_s$ takes values in $\{0,1\}$. Thus an induction gives an optimal strategy taking values in $\{0,1\}$.
\end{proof}

The above is classical in stochastic analysis (see e.g.\ \cite{ps6}) but our proof via the convex relaxation \eqref{ros} seems new. The process $(R_t)_{t=0}^T$ is {\em Markov} if, for all $t$, $(R_{t+1},\dots,R_T)$ is $R_t$-conditionally independent of $\F_t$.

\begin{remark}
If $R\in L^1$ is a Markov process, then $S_{t} = \psi_t(R_t)$ for some measurable function $\psi_t$. In particular, the optimal stopping times $\tau\in\T$ are characterized by the condition $R_\tau=\psi_\tau(R_\tau)$.
\end{remark}

\begin{proof}
Clearly, $S_{T+1} :=0$ is of the required form. Assume that $S_{t+1}=\psi_{t+1}(R_{t+1})$ so that $S_t=\max\{R_t,E_t [\psi_{t+1}(R_{t+1})]\}$. By the conditional independence and Doob-Dynkin lemma \cite[Lemma 1.13]{kal2}, $E_t [\psi_{t+1}(R_{t+1})]=E^{\sigma(R_t)} [\psi_{t+1}(R_{t+1})]=\hat\psi_t(R_t)$ for some measurable function $\hat\psi_t$. Defining $\psi_t(x):=\max\{x,\hat\psi_t(x)\}$, we get $S_t=\psi_t(R_t)$, so the claim follows by induction.
\end{proof}

\subsection{Optimal control}\label{sec:ocdp}

Consider the optimal control problem 
\begin{equation}\label{oc}\tag{OC}
\begin{aligned}
&\minimize\quad & & E\left[\sum_{t=0}^{T} L_t(X_t,U_t)\right]\quad\ovr\ (X,U)\in\N,\\
&\st\quad & & \Delta X_{t}=A_t X_{t-1} +B_t U_{t-1}+W_t\quad t=1,\dots,T
\end{aligned}
\end{equation}
where the {\em state} $X$ and the {\em control} $U$ are processes with values in $\reals^N$ and $\reals^M$, respectively, $A_t$ and $B_t$ are $\F_t$-measurable random matrices, $W_t$ is an $\F_t$-measurable random vector and the functions $L_t$ are convex normal integrands. The linear constrains in \eqref{oc} are called the {\em system equations}. Problem \eqref{oc} fits the general framework with $x=(X,U)$ and 
\[
h(x,\omega)=\sum_{t=0}^{T}L_t(X_t,U_t,\omega) + \sum_{t=1}^{T}\delta_{\{0\}}(\Delta X_{t}- A_t(\omega)X_{t-1}-B_t(\omega)U_{t-1}-W_t(\omega)).
\]

The special structure in \eqref{oc} allows us to express the solution $(h_t)_{t=0}^T$ of the general Bellman equations \eqref{be} in terms of normal integrands $J_t:\reals^N\times\Omega\to\ereals$ and $I_t:\reals^N\times\reals^M\times\Omega\to\ereals$ that solve the following ``dynamic programming'' equations 
\begin{align}
\begin{split}\label{be_OC}
J_{T+1} &= 0\\
I_{t+1}(X_t,U_t) &= J_{t+1}(X_{t}+A_{t+1} X_{t}+B_{t+1} U_{t}+W_{t+1})\\
J_t(X_t) &= \inf_{U_t\in\reals^{M}} E_t(L_t+ I_{t+1})(X_t,U_t).
\end{split}
\end{align}
Note that $J_t$ is a function only of $x_t$ and $\omega$ while the functions $h_t$ in the general Bellman equations \eqref{be} may depend on the whole path $x^t$ of $x$ up to time $t$. The functions $J_t$ are often called the ``value functions'' or ``cost-to-go functions''. This terminology is well justified by \thref{thm:dpOC} below. 
Informally, \eqref{be_OC} can be written in the more familiar form
\begin{align*}
J_{T+1} &= 0\\
J_t(X_t) &= \inf_{U_t\in\reals^{M}} E_t[L_t(X_t,U_t)+J_{t+1}(X_{t}+A_{t+1} X_{t}+B_{t+1}U_{t} + W_{t+1})],
\end{align*}
which are the equations studied in e.g.\ \cite[Section~2.2]{ber76}, \cite[Section~1.2]{bs78} and \cite[Proposition~4.12]{cccd15} when applied to the convex case. We have introduced the functions $I_t$ to clarify that the conditional expectations are taken here in the sense of normal integrands. This resolves many of the measurability problems that arise in earlier formulations. The optimality conditions and the existence results for the convex control problems in \thref{thm:dpOC,thm:dpOC2} below seem new in the presented generality.

Essentially, our formulation of \eqref{oc} in terms of normal integrands is what \cite[Section~1.2.II]{bs78} calls ``semicontinuous models''. In \cite{bs78}, these models were, however, interpreted quite narrowly with the exclusion of even the classical linear quadratic model in \thref{ex:lqc}. The discrepancy seems to have come from the belief that ``in the usual stochastic programming model, the controls cannot influence the distribution of future states''; see page 12 of \cite{bs78}. It is clear in \eqref{oc}, however, that the controls do influence the distribution of future states. What is crucial for convexity is that the state equations are affine. In nonconvex dynamic programming recursions of~\cite{evs76,ppr16}, even this assumption can be relaxed.

Compared to the general Bellman equations \eqref{be}, \eqref{be_OC} provides a significant dimension-reduction with respect to time: the optimal control does not depend on the past states. This is often referred to as the ``dynamic programming principle''. This reduction is essentially due to the time-separable structure of \eqref{oc} where, conditionally on the current state, the future is ``independent'' of the previous states. We will see a further dimension-reduction with respect to scenarios when the random elements in the problem exhibit certain form of independence; see \thref{ex:scIndep} below. Under appropriate conditions, the functions $J_t$ in \eqref{be_OC} turn out to be deterministic.

When applied to \eqref{oc}, \thref{thm:dp0} gives the following.

\begin{theorem}\thlabel{thm:dpOC0}
Assume that \eqref{oc} is feasible, $L_t$ are lower bounded and that $(I_t,J_t)_{t=0}^T$ is a solution of \eqref{be_OC}. Then the optimum value of the optimal control problem coincides with that of
\begin{equation*}
\begin{aligned}
&\minimize\quad & & E\left[\sum_{s=0}^{t-1} (E_t L_s)(X_s,U_s)+J_t(X_t)\right]\quad\ovr\ (X^t,U^t)\in\N^t,\\
&\st\quad & & \Delta X_{s}= A_s X_{s-1}+B_s U_{s-1} +W_s\quad s=1,\dots, t
\end{aligned}
\end{equation*}
for all $t=0,\dots,T$ and, moreover, an $(\bar X,\bar U)\in\N$ solves \eqref{oc} if and only if it satisfies the system equations and
\[
\bar U_t\in\argmin_{U_t\in\reals^{M}} E_t(L_t+I_{t+1})(\bar X_{t},\bar U_t).
\]
for all $t=0,\ldots,T$. If
\[
N_t(\omega)= \{(X_{t},U_{t})\in\reals^{n_{t}}\mid X_{t}=0,\ (E_t(L_{t}+I_{t+1}))^\infty(0,U_{t},\omega)\le 0\}
\]
is linear-valued for all $t=0,\ldots,T$, then there exists an optimal $x\in\N$ with $x_t\perp N_t$ almost surely.
\end{theorem}

Rather than proving \thref{thm:dpOC0} directly, we will prove the following more general result the proof of which is based on \thref{thm:dpLb}. In addition to $p\in\N^\perp$, the assumptions involve $y_t\in L^1(\reals^N)$. Throughout, we set $y_0:=y_{T+1}:=0$ and $X_{T+1}:=0$.

\begin{theorem}\thlabel{thm:dpOC}
Assume that \eqref{oc} is feasible, $(I_t,J_t)_{t=0}^T$ is an $L$-bounded solution of \eqref{be_OC} and that there exists $p\in\N^\perp$, $y \in L^1$ and $m_t\in L^1$ such that  $A_t^* y_t, B_t^* y_t$ and $W_t\cdot y_t$ are integrable for all $t$ and
\begin{enumerate}
\item $E\sum_{t=0}^T L_t(X_t,U_t)=E\sum_{t=0}^T[L_t(X_t,U_t)-(X_t,U_t)\cdot p_t]$ for all $(X,U)\in\N$ satisfying the system equations,
\item $L_t(X_t,U_t) \ge (X_t,U_t)\cdot p_t - X_t\cdot\Delta y_{t+1} - (A_{t+1} X_{t}+B_{t+1} U_{t} +W_{t+1})\cdot y_{t+1}-m_t$ almost surely for all $t$.
\end{enumerate}
Then the optimum value of the optimal control problem coincides with that of
\begin{equation*}
\begin{aligned}
&\minimize\quad & & E\left[\sum_{s=0}^{t-1} (E_t L_s)(X_s,U_s)+J_t(X_t)\right]\quad\ovr\ (X^t,U^t)\in\N^t,\\
&\st\quad & & \Delta X_{s}= A_s X_{s-1}+B_s U_{s-1} +W_s\quad s=1,\dots, t
\end{aligned}
\end{equation*}
for all $t=0,\dots,T$ and, moreover, an $(\bar X,\bar U)\in\N$ solves \eqref{oc} if and only if it satisfies the system equations and
\[
\bar U_t\in\argmin_{U_t\in\reals^{M}} E_t(L_t+I_{t+1})(\bar X_{t},\bar U_t).
\]
for all $t=0,\ldots,T$. If
\[
N_t(\omega)= \{(X_{t},U_{t})\in\reals^{n_{t}}\mid X_{t}=0,\ (E_t(L_{t}+I_{t+1}))^\infty(0,U_{t},\omega)\le 0\}
\]
is linear-valued for all $t=0,\ldots,T$, then there exists an optimal $x\in\N$ with $x_t\perp N_t$ almost surely.
\end{theorem}

\begin{proof}
Let $k(x,\omega):=h(x,\omega)-x\cdot p(\omega)$. Condition 1 means that $Ek=Eh$ on $\N$. For any $(X,U)$ satisfying the system equations, the lower bound in 2 can be written as
\[
L_t(X_t,U_t) \ge (X_t,U_t)\cdot p_t - X_t\cdot\Delta y_{t+1} - \Delta X_{t+1}\cdot y_{t+1}-m_t.
\]
Summing up, shows that $k$ is lower bounded. All the claims follow from \thref{thm:dpLb} once we show that $(h_t)_{t=0}^T$ given by
\begin{align}
  \begin{split}\label{eq:beOC}
    h_t(x^t,\omega) &:= \sum_{s=0}^{t-1}(E_tL_s)(X_{s},U_s,\omega) + E_{t}(L_{t}+I_{t+1})(X_{t},U_{t})\\
    &\quad+  \sum_{s=1}^t\delta_{\{0\}}(-\Delta X_{s}+A_s X_{s-1}+B_s U_{s-1} +W_s)
  \end{split}
\end{align}
is an $L$-bounded solution of \eqref{be}. Indeed, by \thref{thm:dpLb}, the optimum value of \eqref{oc} can then be expressed as
\begin{align*}
  \inf_{x^t\in\N^t}Eh_t(x^t) &= \inf_{x^t\in\N^t}\{\sum_{s=0}^{t-1}(E_tL_s)(X_{s},U_s,\omega) + E_{t}(L_{t}+I_{t+1})(X_{t},U_{t})\mid \\
  &\qquad\qquad\qquad\qquad \Delta x_s = A_s X_{s-1}+B_sU_{s-1}+W_s\ s=1,\ldots,t\}\\
  &= \inf_{x^t\in\N^t}\{\sum_{s=0}^{t-1}(E_tL_s)(X_{s},U_s,\omega) + J_t(X_t)\mid \\
  &\qquad\qquad\qquad\qquad \Delta x_s = A_s X_{s-1}+B_sU_{s-1}+W_s\ s=1,\ldots,t\},
\end{align*}
where the second inequality holds by \cite[Theorem 14.60]{rw98}, Also, if \eqref{eq:beOC} holds, then by \thref{thm:dpLb}, an $(\bar X,\bar U)\in\N$ solves \eqref{oc} if and only if $(\bar X_t,\bar U_t)$ minimizes \eqref{eq:beOC} almost surely.

Assume that \eqref{eq:beOC} solve \eqref{be} from $t$ onwards. Then
\begin{align}
\tilde h_{t-1}(x^{t-1}) &= \inf_{x_t\in\reals^{n_t}} h_t(x^{t-1},x_t)\nonumber\\
&=\sum_{s=0}^{t-1}(E_tL_s)(X_{s},U_s)\nonumber\\
&\quad + \inf_{(X_t,U_t)}\{ E_{t}(L_{t}+I_{t+1})(X_{t},U_{t})\mid X_{t}= X_{t-1}+A_t X_{t-1}+B_t U_{t-1} +W_t\}\nonumber\\
&\quad +\sum_{s=1}^{t-1}\delta_{\{0\}}(-\Delta X_{s}+A_s X_{s-1}+B_s U_{s-1} +W_s)\nonumber\\
&=\sum_{s=0}^{t-1}(E_tL_s)(X_{s},U_s)+J_t(X_{t-1}+A_{t} X_{t-1}+B_{t} U_{t-1} +W_t)\nonumber\\
&\quad+\sum_{s=1}^{t-1}\delta_{\{0\}}(-\Delta X_{s}+A_s X_{s-1}+B_s U_{s-1} +W_s)\nonumber\\
&=\sum_{s=0}^{t-1}(E_tL_s)(X_{s},U_s)+I_t(X_{t-1},U_{t-1})\nonumber\\
&\quad+\sum_{s=1}^{t-1}\delta_{\{0\}}(-\Delta X_{s}+A_s X_{s-1}+B_s U_{s-1} +W_s)\label{eq:tildehOC}
\end{align}
By \thref{thm:condoperations} again,
\begin{align*}
h_{t-1}(x^{t-1}) &=E_{t-1} \tilde h_{t-1}(x^{t-1})\\
& = \sum_{s=0}^{t-2}(E_{t-1}L_s)(X_{s},U_s,\omega) + E_{t-1}(L_{t-1}+I_{t})(X_{t-1},U_{t-1})\\
&\quad +  \sum_{s=1}^{t-1}\delta_{\{0\}}(-\Delta X_{s}+A_s X_{s-1}+B_s U_{s-1} +W_s).
\end{align*}
We have $h_T=E_Th$, by \thref{thm:condoperations}. Thus, by induction, the normal integrands $h_t$ given by \eqref{eq:beOC} solve \eqref{be}.
\end{proof}

\begin{remark}\thlabel{rem:cond1dual}
The lower bounds in \thref{thm:dpOC} hold if and only if 
\[
E[L^*_t(p_t - (\Delta y_{t+1}+A^*_{t+1}y_{t+1},B^*_{t+1}y_{t+1}))-W_{t+1}\cdot y_{t+1}]<\infty,
\]
for all $t$. Here $y_0 := y_{T+1} := 0$, $A_{T+1}:=0$ and $B_{T+1}:=0$. The assumption is closely connected with the feasibility of a problem dual to \eqref{oc}; see \cite{pp22b}.
\end{remark}

\begin{proof}
By Fenchel's inequality, for feasible $(X,U)$,
\begin{align*}
&L_t(X_t,U_t)+L^*_t(p_t - (\Delta y_{t+1}+A^*_{t+1}y_{t+1},B^*_{t+1}y_{t+1}))-W_{t+1}\cdot y_{t+1}\\
 &\ge (X_t,U_t)\cdot p_t -X_t\cdot(\Delta y_{t+1}+A^*_{t+1}y_{t+1})  -U_t\cdot(B^*_{t+1}y_{t+1})-W_{t+1}\cdot y_{t+1} \\
 &=(X_t,U_t)\cdot p_t +X_t\cdot y_t -(X_t+A_{t+1} X_t+B_{t+1} U_t+W_{t+1})\cdot y_{t+1}\\
 &=(X_t,U_t)\cdot p_t +X_t\cdot y_t -(X_{t+1})\cdot y_{t+1}\\
 &=(X_t,U_t)\cdot p_t - X_t\cdot\Delta y_{t+1} - \Delta X_{t+1}\cdot y_{t+1},
\end{align*}
which gives the equivalence.
\end{proof}

\thref{thm:belb} gives the following existence result for \eqref{be_OC}.

\begin{theorem}\thlabel{thm:dpOC20}
Assume that $L_t$ are lower bounded and that the set
\[
\left\{(X,U)\in\N \midb \sum_{t=0}^TL^\infty_t(X_t,U_t)\le 0,\ \Delta X_{t}= A_t X_{t-1}+B_t U_{t-1}\right\}
\]
is linear. Then \eqref{be_OC} has a unique solution $(J_t,I_t)_{t=0}^{T+1}$, each $J_t$ and $I_t$ is lower bounded and
\begin{equation*}\label{eq:OCrec}
N_t(\omega):=\{(X_{t},U_{t})\in\reals^{n_{t}}\mid X_{t}=0,\ (E_t(L_{t}+I_{t+1}))^\infty(0,U_t,\omega)\le 0\}
\end{equation*}
is linear-valued for all $t$.
\end{theorem}

Rather than proving \thref{thm:dpOC20} directly, we will prove the following more general result the proof of which is based on \thref{thm:be}.

\begin{theorem}\thlabel{thm:dpOC2}
Assume that there exists $p\in\N^\perp$, $y \in L^1$ and $m_t\in L^1$ such that
\begin{enumerate}
\item the set
\[
\left\{(X,U)\in\N \midb \sum_{t=0}^T (L^\infty_t(X_t,U_t)-p_t\cdot(X_t,U_t))\le 0,\ \Delta X_{t}= A_t X_{t-1}+B_t U_{t-1}\right\}
\]
is linear,
\item $L_t(X_t,U_t) \ge (X_t,U_t)\cdot p_t - X_t\cdot\Delta y_{t+1} - (A_{t+1} X_{t}+B_{t+1} U_{t} +W_{t+1})\cdot y_{t+1}-m_t$ almost surely for all $t$.
\item $A_t^*y_t$, $B_t^*y_t$ and $W_t\cdot y_t$ are integrable and $E_t[A_t^*y_t]=A_t^*E_ty_t$, $E_t[B_t^*E_t y_t]=B_t^*E_t y_t$ and $E_t[W_t\cdot y]=W_t\cdot E_t y_t$ for all $t$.
\end{enumerate}
Then \eqref{be_OC} has a unique solution $(J_t,I_t)_{t=0}^{T+1}$, each $J_t$ and $I_t$ is L-bounded and
\begin{equation*}\label{eq:OCrec}
N_t(\omega):=\{(X_{t},U_{t})\in\reals^{n_{t}}\mid X_{t}=0,\ (E_t(L_{t}+I_{t+1}))^\infty(0,U_t,\omega)\le 0\}
\end{equation*}
is linear-valued for all $t$.
\end{theorem}

\begin{proof}
Summing up the lower bounds in 2 shows that $k(x,\omega):=h(x,\omega)-x\cdot p(\omega)$ is lower bounded while 1 means that $\{x\in\N\mid k^\infty(x)\le 0\}$ is a linear space. By \thref{thm:be}, \eqref{be} has a unique solution $(h_t)_{t=0}^T$ for $h$ and
\[
N_t(\omega):=\{x_t\in\reals^{n_t} \mid h_t^\infty(x^t,\omega)\le 0,\ x^{t-1}=0\}
\]
is linear-valued for all $t$. 

Assume that $(I_{t'},J_{t'})_{t'=t+1}^{T+1}$ are normal integrands satisfying \eqref{be_OC} from time $t+1$ onwards, and that
\begin{equation*}\label{QqOC}
\begin{aligned}
  \tilde h_{t}(x^{t}) &= \sum_{s=0}^t(E_{t+1}L_s)(X_s,U_s) + I_{t+1}(X_t,U_t) \\
  &\qquad + \sum_{s=1}^{t}\delta_{\{0\}}(\Delta X_s-A_sX_{s-1}+B_sU_{s-1}+W_s),\\
J_{t+1}(X_{t+1}) & \ge X_{t+1}\cdot E_{t+1}y_{t+1} - \sum_{t'=t+1}^TE_{t+1}m_{t'}.
\end{aligned}
\end{equation*}
Here $E_{T+1}$ is defined as the identity mapping so the above hold for $t=T$. The lower bound gives
\begin{align}
I_{t+1}(X_t,U_t) &=J_{t+1}(X_t+A_{t+1}X_t+B_{t+1} U_t +W_{t+1})\nonumber\\
&\ge (X_t+A_{t+1} X_t+B_{t+1} U_t +W_{t+1})\cdot E_{t+1}y_{t+1} - \sum_{t'=t+1}^TE_{t+1}m_{t'}\nonumber\\
&= X_t\cdot E_{t+1}y_{t+1}+X_t\cdot A_{t+1}^*E_{t+1}y_{t+1}+U_t\cdot B_{t+1}^*E_{t+1}y_{t+1}\nonumber\\
&\qquad +W_{t+1}\cdot E_{t+1}y_{t+1} - \sum_{t'=t+1}^TE_{t+1}m_{t'}.\label{eq:Ilb}
\end{align}
Under conditions 2 and 3, all the terms in the expression of $\tilde h_t$ above are L-bounded, by \thref{thm:existce}. By \thref{thm:condoperations,thm:tower},
\begin{align*}
  h_t(x^t) &= (E_t\tilde h_t)(x^t)\\
  &= \sum_{s=0}^{t-1}(E_tL_s)(X_s,U_s) + E_t(L_t+I_{t+1})(X_t,U_t)\\
 &\quad+  \sum_{s=1}^t\delta_{\{0\}}(\Delta X_s-A_sX_{s-1}+B_sU_{s-1}+W_s).
\end{align*}
In particular,
\begin{align*}
 N_t(\omega) &=\{x_t\in\reals^{n_t}\mid h^\infty_t(x_t)\le 0,\ x^{t-1}=0\}\\
 &=\{(X_t,U_t)\in\reals^{n_t}\mid X_t=0,\ (E_t(L_t+I_{t+1}))^\infty(0,U_t)\le 0\},
\end{align*}
which is linear by \thref{thm:be}. Thus, by \thref{thm:ip}, the functions
\[
J_t(X_t):=\inf_{U_t}E_t(L_t+I_{t+1})(X_t,U_t)
\]
and
\[
I_t(X_{t-1},U_{t-1}):=J_t(X_{t-1}+A_{t} X_{t-1}+B_{t} U_{t-1} +W_t)
\]
are convex normal integrands. As in \eqref{eq:tildehOC},
\begin{align*}
\tilde h_{t-1}(x^{t-1}) &=\sum_{s=0}^{t-1}(E_tL_s)(X_s,U_s)+J_t(X_{t-1}+A_tX_{t-1}+B_tU_{t-1} +W_t)\\
&\quad+\sum_{s=1}^{t-1}\delta_{\{0\}}(-\Delta X_{s}+A_s X_{s-1}+B_s U_{s-1} +W_s)\\
&=\sum_{s=0}^{t-1}(E_tL_s)(X_{s},U_s)+I_t(X_{t-1},U_{t-1})\\
&\quad+\sum_{s=1}^{t-1}\delta_{\{0\}}(-\Delta X_{s}+A_s X_{s-1}+B_s U_{s-1} +W_s).
\end{align*}
By \thref{lem:cp1,ex:celq}, conditions 2 and 3 give
\begin{align*}
  (E_{t+1}L_t)(X_t,U_t) &\ge (X_{t},U_{t})\cdot E_{t+1}p_{t} - X_{t}\cdot E_{t+1}[\Delta y_{t+1}]\nonumber \\
  &\quad- X_{t}\cdot A_{t+1}^*E_{t+1}y_{t+1} -U_{t}\cdot B_{t+1}^*E_{t+1}y_{t+1}]\nonumber\\
    &\quad-W_{t+1}\cdot E_{t+1}y_{t+1}-E_{t+1}m_{t}.\label{eq:ELlb}
\end{align*}
Combining this with \eqref{eq:Ilb} gives
\begin{align*}
  (E_{t+1}L_t+I_{t+1})(X_t,U_t) &\ge (X_{t},U_{t})\cdot E_{t+1}p_{t} + X_t\cdot E_{t+1}y_t - \sum_{t'=t}^TE_{t+1}m_{t'}.
\end{align*}
By \thref{lem:cp1,ex:celq},
\[
E_t(L_t+I_{t+1})(X_t,U_t) \ge X_t\cdot E_ty_t - \sum_{t'=t}^TE_tm_{t'}.
\]
and thus,
\[
J_t(X_t)\ge X_t\cdot E_ty_t - \sum_{t'=t}^TE_tm_{t'}.
\]
The claim thus follows by induction on $t$.
\end{proof}

\begin{example}\thlabel{ex:OC2}
Assumptions of \thref{thm:dpOC,thm:dpOC2} hold if \eqref{oc} is feasible,
\begin{enumerate}
\item $\{(X,U)\in\N \mid \sum_{t=0}^T L^\infty_t(X_t,U_t)\le 0,\ \Delta X_{t}= A_t X_{t-1}+B_t U_{t-1}\}$ is  linear.
\item There exists $p\in\N^\perp$ and $\epsilon>0$ such that for every $\lambda\in(1+\epsilon,1-\epsilon)$ there exist $y\in L^1$ and $m\in L^1$ such that $A_t^*y_t, B_t^* y_t$, $y_t\cdot W_t$ are integrable and
\begin{align*}
  L_t(X_t,U_t) &\ge \lambda(X_t,U_t)\cdot p_t - X_t\cdot\Delta y_{t+1} - \Delta X_{t+1}\cdot y_{t+1}-m_t
\end{align*}
for all feasible $(X,U)$ and for all $t$. Here $X_{t+1}:=y_0:=y_{T+1}:=0$.
\end{enumerate}
\end{example}
\begin{proof}
The second assumption clearly implies the lower bound in both theorems. Summing up the lower bounds in 2, we see that \thref{2lambda} is applicable, which implies the rest of the assumptions.
\end{proof}

\begin{remark}\thlabel{rem:qfactor}
Denoting $Q_t:=E_t(L_t+I_{t+1})$, we can write \eqref{be_OC} as
\begin{align*}
  Q_{T+1}&:=0,\\
  I_{t}(X_{t-1},U_{t-1}) &:= \inf_{U_t\in\reals^M}Q_t(X_{t-1}+A_{t} X_{t-1}+B_{t} U_{t-1} +W_t,U_t),\\
  Q_{t-1} &:=E_{t-1}(L_{t-1}+I_t)
\end{align*}
and the optimality condition as
\[
U_t\in\argmin_{U_t\in\reals^M}Q_t(X_t,U_t).
\]
If $(Q_t)_{t=0}^T$ satisfies the above, then the functions
\[
J_t(X_t):=\inf_{U_t\in\reals^M}Q_t(X_t,U_t)
\]
satisfy \eqref{be_OC}. Formulation of the dynamic programming recursion in terms of the $Q$-functions is often used in the context of reinforcement learning e.g.\ in \cite{ber19}.
\end{remark}

\begin{example}[Conditional independence]\thlabel{ex:scIndep}
Assume that $L_t$ are L-bounded, that $(J_t)_{t=0}^\infty$ and $(I_t)_{t=0}^\infty$ are an $L$-bounded solution of \eqref{be_OC} and that there is a sequence of sigma-algebras $(\H_t)_{t=0}^T$ such that
\begin{enumerate}
\item
   $L_t$ is $\H_t$-conditionally independent of $\F_t$ and $\sigma(A_{t+1},B_{t+1},W_{t+1})$ is $\H_{t+1}\vee\H_{t}$-conditionally independent of $\F_t$.
\item
  $\H_{t+1}$ is $\H_t$-conditionally independent of $\F_t$.
\end{enumerate}
Then $J_t$ is $\H_t$-measurable and
\[
E_t(L_t+I_{t+1}) = E^{\H_t}(L_t+I_{t+1}).
\]
In particular, if $(\H_t)_{t=0}^{T+1}$ are mutually independent and  $A_{t}$, $B_{t}$ and $W_{t}$ are $\H_t$-measurable , then each $E_t I_{t+1}$ is deterministic. If, in addition, $L_t$ is independent of $\F_t$, then $J_t$ is deterministic.
\end{example}

\begin{proof}
Assume that $J_{t+1}$ is $\H_{t+1}$-measurable. Then
\[
I_{t+1}(X_{t},U_{t})=J_{t+1}(X_{t}+A_{t+1}X_t+B_{t+1}U_{t}+W_{t+1})
\]
is $\sigma(A_{t+1},B_{t+1},W_{t+1})\vee \H_{t+1}$-measurable. By \cite[Proposition 6.8]{kal2}, conditions 1 and 2 imply that $I_{t+1}$ is $\H_{t}$-conditionally independent of $\F_{t}$. By \thref{thm:cini},
\[
E_{t}(L_{t}+I_{t+1})= E^{\H_{t}} (L_{t}+I_{t+1}),
\]
so $J_{t}$ is $\H_{t}$ measurable, by \thref{thm:ip}. By \thref{thm:cini}, condition 1 implies that $J_T$ is $\H_T$-measurable, so the claim follows from induction.
 
Under the additional assumptions, $I_{t+1}$ is $\H_{t+1}$-measurable, so mutual independence of $\H_t$ gives, by \thref{thm:cini}, that
\[
E_{t}I_{t+1}= E^{\H_{t}} I_{t+1}
\]
is deterministic. When, in addition, $L_t$ is independent of $\F_t$,
\[
E_{t-1}(L_{t-1}+I_t)
\]
is deterministic as well, and so too is $J_t$.
\end{proof}

\begin{example}[Linear-quadratic control]\thlabel{ex:lqc}
Consider \eqref{oc} in the case
\[
L_t(X_t,U_t) = \frac{1}{2}X_t\cdot Q_tX_t+\frac{1}{2}U_t\cdot R_tU_t,
\]
where $Q_t\in L^1(\reals^{N\times N})$ and $R_t\in L^1(\reals^{M\times M})$ are random symmetric positive semi-definite matrices. Assume that the recursion
\begin{align*}
  K_T&:=E_T[Q_T],\\
  K_{t-1} & :=E_{t-1}[Q_{t-1}+(I+A_t)^*K_t(I+A_t)]\\
  &\quad -\frac{1}{2}E_{t-1}[(I+A_t)^*K_tB_t](E_{t-1}[R_{t-1} + B_t^*K_tB_t])^{-1}E_{t-1}[B_t^*K_t(I+A_t)]
\end{align*}
is well-defined, $W_t\cdot K_tA_t$ and $W_t\cdot K_tA_t$ are integrable, $W_t$ has zero mean and is independent of $\F_{t-1}$ and of $A_{t'}$, $B_{t'}$, $Q_{t'}$ and $R_{t'}$ for $t'\ge t$. Then
\[
J_t(X_t) = \frac{1}{2}X_t\cdot K_tX_t + \frac{1}{2}\sum_{t'=t+1}^TE_t[W_{t'}\cdot Q_{t'}W_{t'}]
\]
and the optimal control is given by $\bar U_t = - \Lambda_tX_t$, where
\[
\Lambda_t=(E_t[R_t + B_{t+1}^*K_{t+1}B_{t+1}])^{-1}E_t[B_{t+1}^*K_{t+1}(I+A_{t+1})].
\]

Note that if, in addition, $Q_t$, $R_t$, $A_{t+1}$ and $B_{t+1}$ are independent of $\F_t$, then the matrices $K_t$ and $\Lambda_t$ are deterministic with
\begin{align*}
  K_T&:=E[Q_T],\\
  K_{t-1} & :=E[Q_{t-1}+(I+A_t)^*K_t(I+A_t)]\\
  &\quad -\frac{1}{2}E[(I+A_t)^*K_tB_t](E[R_{t-1} + B_t^*K_tB_t])^{-1}E[B_t^*K_t(I+A_t)]
\end{align*}
and
\[
\Lambda_t = (E[R_t] + E[B_{t+1}]^*K_{t+1}E[B_{t+1}])^{-1}E[B_{t+1}]^*K_{t+1}E[I+A_{t+1}].
\]

The backward recursion for $(K_t)_{t=0}^T$ are known as Riccati equations. The matrices $K_t$ are well-defined e.g.\ when $Q_t$, $R_t$, $A_t$ and $B_t$ are all deterministic and $R_t$ are positive definite.
\end{example}

\begin{proof}
By \thref{ex:celq},
\[
E_tL_t(X_t,U_t)=\frac{1}{2}X_t\cdot E_t[Q_t]X_t+\frac{1}{2}U_t\cdot E_t[R_t]U_t,
\]
for all $t$, so
\[
J_T(X_T)=\frac{1}{2}X_T\cdot E_T[Q_T]X_T.
\]
Assume that the expression for $J_{t'}$ in the claim is valid for $t'=t,\ldots,T$. We get
\begin{align*}
  I_t(X_{t-1},U_{t-1}) &= J_t((I+A_t)X_{t-1}+B_tU_{t-1}+W_t)\\
  &=\frac{1}{2}X_{t-1}\cdot(I+A_t)^*K_t(I+A_t)X_{t-1} + \frac{1}{2}U_{t-1}\cdot B_t^*K_tB_tU_{t-1}\\
  &\quad + \frac{1}{2}W_t\cdot K_tW_t + X_{t-1}\cdot (I+A_t)^*K_tB_tU_{t-1}\\
  &\quad + W_t\cdot K_t(X_{t-1}+A_tX_{t-1}+B_tU_{t-1}) + \frac{1}{2}\sum_{t'=t+1}^TE_t[W_{t'}\cdot Q_{t'}W_{t'}].
\end{align*}
When $W_t$ has zero mean and is independent of $\F_{t-1}$, $Q_{t'}$, $R_{t'}$, $A_{t'}$ and $B_{t'}$ for $t'\ge t$ and $W_t$ is independent of $\F_{t-1}$, \thref{ex:celq} gives
\begin{align*}
  E_{t-1}(L_{t-1}+I_t)(X_{t-1},U_{t-1}) &=\frac{1}{2}X_{t-1}\cdot E_{t-1}[Q_{t-1}+(I+A_t)^*K_t(I+A_t)]X_{t-1}\\
  &\quad + \frac{1}{2}U_{t-1}\cdot E_{t-1}[R_{t-1} + B_t^*K_tB_t]U_{t-1}\\
  &\quad + X_{t-1}\cdot E_{t-1}[(I+A_t)^*K_tB_t]U_{t-1} + \frac{1}{2}\sum_{t'=t}^TE_{t-1}[W_{t'}\cdot Q_{t'}W_{t'}].
\end{align*}
Thus,
\begin{align*}
  J_{t-1}(X_{t-1}) &= \inf_{U_{t-1}\in\reals^M}E_{t-1}(L_{t-1}+I_t)(X_{t-1},U_{t-1})\\
  &= \frac{1}{2}X_{t-1}\cdot E_{t-1}[Q_{t-1}+(I+A_t)^*K_t(I+A_t)]X_{t-1}+ \frac{1}{2}E_{t-1}[W_t\cdot K_tW_t]\\
  &\quad + \inf_{U_{t-1}\in\reals^M}\{\frac{1}{2}U_{t-1}\cdot E_{t-1}[R_{t-1} + B_t^*K_tB_t]U_{t-1} + X_{t-1}\cdot E_{t-1}[(I+A_t)^*K_tB_t]U_{t-1}\}\\
  &=  \frac{1}{2}X_{t-1}\cdot K_{t-1}X_{t-1} + \frac{1}{2}\sum_{t'=t}^TE_t[W_{t'}\cdot Q_{t'}W_{t'}],
\end{align*}
where the infimum is attained by the $\bar U_t$ in the statement. The claim thus follows by induction on $t$.
\end{proof}

\subsection{Problems of Lagrange}\label{sec:lagrange2}

Consider the problem
\begin{equation}\label{lagrange}\tag{$L$}
\minimize\quad E\sum_{t=0}^T K_t(x_t,\Delta x_t)\quad\ovr x\in\N,
\end{equation}
where $x$ is a process of fixed dimension $d$, $K_t$ are convex normal integrands and $x_{-1}:=0$. This fits the general format with
\[
h(x,\omega) = \sum_{t=0}^T K_t(x_t,\Delta x_t,\omega).
\]
Problem \eqref{lagrange} is a special case of a stochastic problem of Bolza studied in \cite{rw83} whose objective contains an additional convex function of $(x_0,x_T)$. This section analyzes the dynamic programming principle for \eqref{lagrange}. The results below seem to be new.

If in the linear stochastic programming model \thref{ex:lp}, the constraints can be grouped as
\begin{equation}\label{eq:classicLP}
T_t\Delta x_t+W_tx_t - b_t\in C_t\quad t=0,\ldots,T
\end{equation}
for given random matrices $T_t$ and $W_t$ and cones $C_t$, then the problem in \thref{ex:lp} is an instance of \eqref{lagrange} with
\[
K_t(\Delta x_t,x_t,\omega) =
\begin{cases}
  c_t(\omega)\cdot x_t & \text{if $T_t(\omega)\Delta x_t+W_t(\omega)x_t - b_t(\omega)\in C_t(\omega)$},\\
  +\infty & \text{otherwise}.
\end{cases}
\]
Note also that the problem of optimal control \eqref{oc} is a special case of \eqref{lagrange} with $x_t=(X_t,U_t)$ and
\[
K_t(x_t,\Delta x_t) =
\begin{cases}
  L_t(X_t,U_t) & \text{if $\Delta X_t = A_tX_{t-1}+B_tU_{t-1}+W_t$},\\
  +\infty & \text{otherwise}.
\end{cases}
\]
Other examples can be found e.g.\ in financial mathematics; see \cite[Example~3.6]{bpp18}.

Much like in the optimal control problem \eqref{oc}, the time-separable structure here allows us to express the solutions of the generalized Bellman equations \eqref{be} in terms of normal integrands $V_t$ and $\tilde V_t$ on $\reals^d\times\Omega$ that solve the following dynamic programming equations
\begin{equation}\label{be_V}
\begin{split}
V_T &= 0,\\
\tilde V_{t-1}(x_{t-1},\omega) &= \inf_{x_t\in\reals^d}\{(E_t K_t)(x_{t},\Delta x_t,\omega) + V_t(x_t,\omega)\},\\
V_{t-1} &= E_{t-1}\tilde V_{t-1}.
\end{split}
\end{equation}
Note that $V_t$ is a function only of $x_t$ and $\omega$ while the functions $h_t$ in the general Bellman equations \eqref{be} may depend on the whole path of $x$ up to time~$t$.

\thref{thm:belb} gives the following existence result for \eqref{lagrange}.

\begin{theorem}\thlabel{thm:dpL0}
Assume that \eqref{lagrange} is feasible, $K_t$ are lower bounded and $(V_t)_{t=0}^T$ is a solution of \eqref{be_V}. Then 
\begin{align*}
\inf\eqref{lagrange} &=\inf_{x^t\in\N^t}E\left[\sum_{s=0}^t (E_t K_s)(x_{s},\Delta x_s) + V_t(x_t)\right]\quad t=0,\ldots,T,
\end{align*}
for all $t=0,\dots,T$ and, moreover, an $\bar x\in\N$ solves \eqref{lagrange} if and only if
\[
x_t\in\argmin_{x_t\in\reals^d}\{(E_t K_t)(x_{t},\Delta x_t) + V_t(x_t)\}\quad \text{a.s.}
\]
for all $t=0,\ldots,T$. If
\[
N_t(\omega)=\{x_t\in\reals^d\mid (E_t K_t)^\infty(x_t,x_t,\omega)+V^\infty_t(x_t,\omega) \le 0\}
\]
is linear-valued for all $t=0,\ldots,T$, then there exists an optimal $x\in\N$ with $x_t\perp N_t$ almost surely.
\end{theorem}

Rather than proving \thref{thm:dpL0} directly, we will prove the following more general result the proof of which is based on \thref{thm:dpLb}. In addition to the ``shadow price of information'' $p\in\N^\perp$, the assumptions in this section involve other ``dual variables'' $y_t\in L^1(\reals^d)$. Throughout, we set $y_{T+1}:=0$.

\begin{theorem}\thlabel{thm:dpL}
Assume that \eqref{lagrange} is feasible, $(V_t)_{t=0}^T$ is an $L$-bounded solution of \eqref{be_V} and that there exists $p\in\N^\perp$, $y \in L^1$ and $m_t\in L^1$ such that
\begin{enumerate}
\item $E\sum K_t(x_t,\Delta x_t)=E\sum[K_t(x_t,\Delta x_t)-x_t\cdot p_t]$ for all $x\in\N$,
\item $K_t(x_{t},\Delta x_t) \ge x_{t}\cdot(p_{t}+\Delta y_{t+1}) + \Delta x_t\cdot y_t - m_t$ almost surely for all $t$.
\end{enumerate}
Then 
\begin{align*}
\inf\eqref{lagrange} &=\inf_{x^t\in\N^t}E\left[\sum_{s=0}^t (E_t K_s)(x_{s},\Delta x_s) + V_t(x_t)\right]\quad t=0,\ldots,T,
\end{align*}
for all $t=0,\dots,T$ and, moreover, an $\bar x\in\N$ solves \eqref{lagrange} if and only if
\[
x_t\in\argmin_{x_t\in\reals^d}\{(E_t K_t)(x_{t},\Delta x_t) + V_t(x_t)\}\quad \text{a.s.}
\]
for all $t=0,\ldots,T$. If
\[
N_t(\omega)=\{x_t\in\reals^d\mid (E_t K_t)^\infty(x_t,x_t,\omega)+V^\infty_t(x_t,\omega) \le 0\}
\]
is linear-valued for all $t=0,\ldots,T$, then there exists an optimal $x\in\N$ with $x_t\perp N_t$ almost surely.
\end{theorem}

\begin{proof}
Summing up the lower bounds in 2 shows that $k(x,\omega):=h(x,\omega)-x\cdot p(\omega)$ is lower bounded while 1 means that $Ek=Eh$ on $\N$. Assume that $(V_t)_{t=0}^T$ is an L-bounded solution of \eqref{be_V}. All the claims follow from \thref{thm:dpLb} once we show that $(h_t)_{t=0}^T$ given by
 \begin{equation*}
h_t(x^t,\omega) := \sum_{s=0}^t(E_tK_s)(x_{s},\Delta x_s,\omega) + V_t(x_t,\omega)
\end{equation*}
is an $L$-bounded solution of \eqref{be}. Assume that $(h_{t'})_{t'=t}^T$ satisfies \eqref{be} from time $t$ onwards. We get
\begin{align*}
\tilde h_{t-1}(x^{t-1},\omega) &:= \inf_{x_t\in\reals^d} h_t(x^{t-1},x_t,\omega)\\
&= \sum_{s=0}^{t-1}(E_tK_s)(x_{s},\Delta x_s,\omega) + \inf_{x_t\in\reals^d}\{(E_tK_t)(x_{t},\Delta x_t,\omega) + V_t(x_t,\omega)\}\\
&= \sum_{s=0}^{t-1}(E_tK_s)(x_{s},\Delta x_s,\omega) + \tilde V_{t-1}(x_{t-1}).
\end{align*}
By \thref{thm:condoperations} again,
\[
h_{t-1}(x^t,\omega) =E_{t-1} \tilde h_{t-1}.
\]
Since $V_T=0$, we have $h_T=E_Th$, by \thref{thm:condoperations}. Thus, by induction, $(h_t)_{t=0}^T$ solves \eqref{be}.
\end{proof}

The following gives sufficient conditions for the existence of solutions to the Bellman equations \eqref{be_V}.

\begin{theorem}\thlabel{thm:dpL20}
Assume that $K_t$ are lower bounded and that 
\[
\{x\in\N\mid \sum_{t=0}^T (K^\infty_t(x_t,\Delta x_t)-x_t\cdot p_t)\le 0\}
\]
is a linear space. Then \eqref{be_V} has a unique solution $(V_t)_{t=0}^T$, each $V_t$ is lower bounded and
\begin{equation}\label{eq:Lagrec}
N_t(\omega):=\{x_t\in\reals^d \mid (E_tK_t)^\infty(x_t,x_t,\omega) + V_t^\infty(x_t,\omega)\le 0\}
\end{equation}
is linear-valued for all $t$.
\end{theorem}

Rather than proving \thref{thm:dpL20} directly, we will prove the following more general result the proof of which is based on \thref{thm:be}.

\begin{theorem}\thlabel{thm:dpL2}
Assume that there exists $p\in\N^\perp$, $y \in L^1$ and $m_t\in L^1$ such that
\begin{enumerate}
\item $\{x\in\N\mid \sum_{t=0}^T (K^\infty_t(x_t,\Delta x_t)-x_t\cdot p_t)\le 0\}$ is a linear space,
\item $K_t(x_{t},\Delta x_t) \ge x_{t}\cdot(p_{t}+\Delta y_{t+1}) + \Delta x_t\cdot y_t - m_t$ almost surely for all $t$.
\end{enumerate}
Then \eqref{be_V} has a unique solution $(V_t)_{t=0}^T$, each $V_t$ is L-bounded and
\begin{equation}\label{eq:Lagrec}
N_t(\omega):=\{x_t\in\reals^d \mid (E_tK_t)^\infty(x_t,x_t,\omega) + V_t^\infty(x_t,\omega)\le 0\}
\end{equation}
is linear-valued for all $t$.
\end{theorem}

\begin{proof}
Summing up the lower bounds in 2 shows that $k(x,\omega):=h(x,\omega)-x\cdot p(\omega)$ is lower bounded while 1 means that $\{x\in\N\mid k^\infty(x)\le 0\}$ is a linear space. By \thref{thm:be}, \eqref{be} has a unique solution $(h_t)_{t=0}^T$ for $h$ and
\begin{equation}\label{eq:Lagrec0}
N_t(\omega):=\{x_t\in\reals^{n_t} \mid h_t^\infty(x^t,\omega)\le 0,\ x^{t-1}=0\}
\end{equation}
is linear-valued for all $t$. 

Assume that $(V_{t'})_{t'=t}^T$ are normal integrands satisfying \eqref{be_V} from time $t$ onwards, and that
\begin{equation}\label{Qqa}
\begin{split}
h_t(x^t) &= \sum_{s=0}^t(E_tK_s)(x_{s},\Delta x_s) + V_t(x_t)\\
V_t(x_t) &\ge - x_t\cdot (E_ty_{t+1}) - E_t\sum_{t'=t+1}^Tm_{t'}
\end{split}
\end{equation}
almost surely. We get
\begin{align}
\tilde h_{t-1}(x^{t-1}) &= \inf_{x_t\in\reals^d} h_t(x^{t-1},x_t)\nonumber\\
&= \sum_{s=0}^{t-1}(E_{t}K_s)(x_{s},\Delta x_s) + \tilde V_{t-1}(x_{t-1}),\label{eq:tildehL}
\end{align}
where
\[
\tilde V_{t-1}(x_{t-1}) := \inf_{x_{t}\in\reals^d}\{(E_tK_t)(x_t,\Delta x_t) + V_t(x_t)\}.
\]
The equation in \eqref{Qqa} and the linearity of \eqref{eq:Lagrec0} imply that \eqref{eq:Lagrec} is linear. Thus, by \thref{thm:ip}, $\tilde V_{t-1}$ is a normal integrand. 

By \thref{lem:cp1,ex:celq}, condition 2 implies 
\[
(E_tK_t) (x_t,\Delta x_t)\ge x_t \cdot E_t\Delta y_{t+1} + \Delta x_t\cdot E_t y_{t} - E_tm_t
\]
Combining with the inequality in \eqref{Qqa} gives
\[
\tilde V_{t-1}(x_{t-1}) \ge - x_{t-1}\cdot E_ty_t - E_t\sum_{t'=t}^Tm_{t'}.
\]
By \thref{thm:existce}, there is a unique normal integrand $V_{t-1}$ such that $V_{t-1}=E_{t-1}\tilde V_{t-1}$. By \thref{lem:cp1,ex:celq}, $V_{t-1}$ satisfies the inequality in \eqref{Qqa}. Taking conditional expectations on both sides of \eqref{eq:tildehOC} and using \thref{thm:condoperations,thm:tower}, we get
\[
h_{t-1}(x^{t-1}) = \sum_{s=0}^{t-1}(E_{t-1}K_s)(x_s,\Delta x_s) + V_{t-1}(x_{t-1}).
\]
For $t=T$, $V_T=0$ so \eqref{Qqa} holds by \thref{thm:condoperations}. The claim thus holds by induction on $t$.
\end{proof}

The following analogue of \thref{2lambda} gives sufficient conditions for the assumptions of  \thref{thm:dpL}.

\begin{example}\thlabel{ex:l2}
Assumptions of \thref{thm:dpL,thm:dpL2} hold if \eqref{lagrange} is feasible,
\begin{enumerate}
\item $\{x\in\N\mid \sum_{t=0}^T K^\infty_t(x_t,\Delta x_t)\le 0\}$ is a linear space,
\item there exists $p\in\N^\perp$ and $\epsilon>0$ such that for every $\lambda\in(1-\epsilon,1+\epsilon)$, there exist $y \in L^1$ and $m_t\in L^1$ with 
\[
K_t(x_{t},\Delta x_t) \ge x_{t}\cdot(\lambda p_{t}+\Delta y_{t+1}) + \Delta x_t\cdot y_t - m_t
\]
for all $t$.
\end{enumerate}
\end{example}

\begin{proof}
The second assumption clearly implies the lower bound in both theorems. Summing up the lower bounds in 2, we see that \thref{2lambda} is applicable, which implies the rest of the assumptions.
\end{proof}

The linearity condition in \thref{ex:l2} holds, in particular, if $K_t^\infty\ge 0$ and $K_t^\infty(x_t,x_t)>0$ for every $t=0,\ldots,T$ and $x_t\ne 0$. Indeed, in this case
\[
\{x\in\reals^n\mid \sum_{t=0}^T K^\infty_t(x_t,\Delta x_t)\le 0\}=\{0\}
\]
almost surely. The last condition in \thref{ex:l2} will be discussed in \cite{pp22b}.


\begin{example}[Block-diagonal stochastic LP]\thlabel{bdslp}
Consider the stochastic linear programming problem from \thref{ex:lp} and assume that the constraint $Ax-b\in K$ can be written as
\[
T_t\Delta x_t+W_tx_t - b_t\in C_t\quad t=0,\ldots,T.
\]
This is an instance of \eqref{lagrange} with
\[
K_t(\Delta x_t,x_t,\omega) =
\begin{cases}
  c_t(\omega)\cdot x_t & \text{if $T_t(\omega)\Delta x_t+W_t(\omega)x_t - b_t(\omega)\in C_t(\omega)$},\\
  +\infty & \text{otherwise}.
\end{cases}
\]
If $T_t$, $W_t$, $b_t$ and $C_t$ are $\F_t$-measurable, then $K_t$ is an $\F_t$-measurable convex normal integrand and the dynamic programming recursion \eqref{be_V} can be written as  
\begin{align*}
V_T &= 0,\\
\tilde V_{t-1}(x_{t-1},\omega) &= \inf_{x_t\in\reals^d}\{c_t(\omega)\cdot x_t + V_t(x_t,\omega)\mid T_t(\omega)\Delta x_t+W_t(\omega)x_t - b_t(\omega)\in C_t(\omega)\},\\
V_{t-1} &= E_{t-1}\tilde V_{t-1}.
\end{align*}
This is the classic formulation for linear stochastic programs. \thref{thm:dpL2} gives sufficient conditions for this to be well-defined. 

Existence of solutions to Bellman equations for linearly constrained problems have been established also in \cite{ols76b}. 
\end{example}

\begin{example}[Conditional independence]\thlabel{ex:lagrangeIndep}
Assume that $K_t$ are L-bounded, that $(V_t)_{t=0}^T$ and $(\tilde V_t)_{t=0}^T$ are an $L$-bounded solution of \eqref{be_V} and that there is a sequence of sigma-algebras $(\H_t)_{t=0}^T$ such that $\H_t\subseteq\F_t$ and
\begin{enumerate}
\item
  $K_t$ is $\H_t$-conditionally independent of $\F_t$,
\item
  $\H_{t+1}$ is $\H_t$-conditionally independent of $\F_t$.
\end{enumerate}
Then $\tilde V_t$ is $\H_{t+1}$-measurable and 
\[
V_t= E^{\H_t}\tilde V_t.
\]
In particular, if $(\H_t)_{t=0}^{T+1}$ are mutually independent, then $V_t$ are deterministic.
\end{example}

\begin{proof}
By definition, $V_T=0$ is $\H_T$-measurable.  Assume now that $V_{t+1}$ is $\H_{t+1}$-measurable. By \thref{thm:cini}, condition 1 implies $E_{t+1} K_{t+1}$ is $\H_{t+1}$-measurable. Thus, $\tilde V_t$ is then $\H_{t+1}$-measurable as well. By Theorem~\ref{thm:cini} again, condition 2 implies $V_t=E^{\H_t}\tilde V_t$, so the first two claims follow by induction. The last claim follows from Theorem~\ref{thm:cini}.
\end{proof}

\begin{remark}
Even under the assumptions of  \thref{thm:dpL,thm:dpL2}, $K_t(x_t,\Delta x_t)$ need not be integrable for an optimal solution $x$ of \eqref{lagrange}. Indeed, defining 
\[
K_t(x_t,\Delta x_t,\omega):=x_t\cdot\Delta y_{t+1}(\omega)+y_t(\omega)\cdot\Delta x_t,
\]
for a process $y\in L^1$ with $y_{T+1}:=0$, the assumptions of \thref{ex:l2} are satisfied with $p=0$. Since $\sum_{t=0}^T K_t(x_t,\Delta x_t,\omega)=0$, any $x\in\N$ is optimal, but $K_t(x_t,\Delta x_t)$ need not be  integrable.
\end{remark}

\subsection{Financial mathematics}\label{sec:fm2}

Let $s=(s_t)_{t=0}^T$ be an adapted $\reals^J$-valued stochastic process describing the unit prices of traded assets in a perfectly liquid financial market. Consider the problem of finding a dynamic trading strategy $z=(z_t)_{t=0}^T$ that provides the ``best hedge'' against the financial liability of delivering a random amount $c\in L^0$ of cash at time $T$. If we measure our risk preferences over random cash-flows with the ``expected shortfall'' associated with a nondecreasing nonconstant convex ``loss function'' $V:\reals\to\ereals$, the problem can be written as
\begin{equation}\label{alm}\tag{$ALM$}
\begin{aligned}
  &\minimize\quad EV\left(c-\sum_{t=0}^{T-1}x_t\cdot\Delta s_{t+1}\right)\quad\ovr\quad x\in\N,\\
  &\st\quad x_t\in D_t\quad t=0,\ldots,T\ a.s.,
\end{aligned}
\end{equation}
where $D_t$ is a random $\F_t$-measurable set describing possible portfolio constraints. We will assume $D_T=\{0\}$, which means that all positions have to be closed at the terminal date. Note that nondecreasing convex loss functions $V$ are in one-to-one correspondence with nondecreasing concave utility functions $U$ via $V(c)=-U(-c)$.

Problem \eqref{alm} is standard in financial mathematics although it is based on quite unrealistic assumptions on the financial market. In particular, it assumes that we can buy and sell arbitrary quantities of all assets at prices given by $s$. It also assumes that one can lend and borrow arbitrary amounts of cash at zero interest rate. Under these assumptions, the random variable $c$ can be thought of as the difference of the claim to be hedged and the initial wealth and the sum in the objective can be interpreted as the proceeds from trading from time $t=0$ to $t=T$. 

Problem \eqref{alm} fits the general framework with
\begin{align*}
h(x,\omega) &= V\left(c(\omega)-\sum_{t=0}^{T-1} x_t\cdot\Delta s_{t+1}(\omega),\omega\right) + \sum_{t=0}^{T-1}\delta_{D_t(\omega)}(x_t,\omega).
\end{align*}

We start by giving sufficient conditions for the main existence results in Section~\ref{sec:Lbexist}. It turns out that, in the absence of portfolio constraints, the linearity condition in \thref{2lambda} becomes the classical {\em no-arbitrage} condition
\begin{equation}\tag{NA}\label{na}
x\in\N,\ \sum_{t=0}^{T-1} x_t\cdot\Delta s_{t+1}\ge 0\ a.s.\implies\sum_{t=0}^{T-1} x_t\cdot\Delta s_{t+1}= 0\ a.s.;
\end{equation}
see \thref{rem:fmlin} below. The lower bound in \thref{2lambda} holds, in particular, if there exists a martingale measure $Q\ll P$ such that
\[
cy,\ V^*(dQ/dP)\in L^1,
\]
$V$ is deterministic and either
\begin{equation*}
 AE_-(V):=\limsup_{u\to-\infty} \frac{uV'(u)}{V(u)}<1\quad\text{or}\quad AE_+(V):=\liminf_{u\to+\infty} \frac{uV'(u)}{V(u)}> 1;
\end{equation*}
see \thref{rem:fmlb}. The above limits are known as ``asymptotic elasticies'' of $V$. The conditions above are satisfied by most familiar loss functions such as the exponential, logarithmic and power functions. 

The problem \eqref{alm} fits the general framework with
\begin{align*}
h(x,\omega) &= V\left(c(\omega)-\sum_{t=0}^{T-1} x_t\cdot\Delta s_{t+1}(\omega),\omega\right) + \sum_{t=0}^{T-1}\delta_{D_t(\omega)}(x_t,\omega).
\end{align*}
As before, we assume that $V$ is a nondecreasing, nonconstant convex normal integrand on $\Omega\times\reals$. This implies, in particular that, the recession function $V^\infty$ of $V$ is nondecreasing and strictly positive on strictly positive reals.

\begin{theorem}[Existence of solutions]\thlabel{thm:alm1}
Assume that
\begin{enumerate}
\item
  there exist $y\in L^0$ and $\epsilon>0$ such that
  \[
cy,\ y\Delta s_t,\ \sigma_{D_t}(E_t[y\Delta s_{t+1}]),\ V^*(\lambda y)\in L^1\quad\forall t,\ \lambda\in[1-\epsilon,1+\epsilon]
\]
\item
  the set $\L=\{x\in \N \mid \sum_{t=0}^{T-1} x_t\cdot\Delta s_{t+1} \ge 0,\ x_t\in D^\infty_t \text{ $P$-a.s.}\}$ is linear.
\end{enumerate}
Then \eqref{alm} admits optimal solutions.
\end{theorem}

\begin{proof}
By \thref{thm:dpLb,thm:dplbc}, it suffices to show that the two conditions in \thref{2lambda} hold. By Fenchel's inequality, $V(u)\ge \lambda uy-V^*(\lambda y)$ and
\[
\delta_{D_t}(x_t)\ge \lambda x_t\cdot(E_t[y\Delta s_{t+1}])-\lambda\sigma_{D_t}(E_t[y\Delta s_{t+1}])
\]
so
\[
h(x) \ge \lambda cy - V^*(\lambda y) - \lambda \sum_{t=0}^{T-1}x_t\cdot[y\Delta s_{t+1}-E_t[y\Delta s_{t+1}] - \lambda\sum_{t=0}^{T-1}\sigma_{D_t}(E_t[y\Delta s_{t+1}]).
\]
Choosing $p_t=E_t[y\Delta s_{t+1}]-y\Delta s_{t+1}$ and  
\begin{align*}
  m &= \max_{\lambda\in[1-\epsilon,1+\epsilon]}\{V^*(\lambda y)+\lambda\sum_{t=0}^{T-1}\sigma_{D_t}(E_t[y\Delta s_{t+1}])-\lambda cy\}
\end{align*}
gives the first condition of \thref{2lambda}. Indeed, by convexity of $V^*$, the maximum in the expression of $m$ is attained scenariowise at $\lambda=1-\epsilon$ or $\lambda=1+\epsilon$ so $m$ is integrable as the pointwise maximum of two integrable functions. By \cite[Theorem~9.3]{rw98},
\begin{align*}
h^\infty(x,\omega) &= V^\infty\left(-\sum_{t=0}^{T-1} x_t\cdot\Delta s_{t+1}(\omega),\omega\right) + \sum_{t=0}^{T-1}\delta_{D_t^\infty(\omega)}(x_t,\omega)
\end{align*}
so condition 2 in \thref{2lambda} means that
\[
\L=\{x\in \N \mid V^\infty\left(-\sum_{t=0}^{T-1} x_t\cdot\Delta s_{t+1}\right) \le 0,\ x_t\in D^\infty_t \text{ $P$-a.s.}\}
\]
is a linear space. Since $V$ is not a constant function, we have $V^\infty(u)> 0$ for $u>0$ and hence 
\[
V^\infty(-\sum_{t=0}^{T-1} x_t\cdot\Delta s_{t+1}) \le 0\quad \Leftrightarrow\quad \sum_{t=0}^{T-1} x_t\cdot\Delta s_{t+1} \ge 0.
\]
Thus,
\[
\L=\{x\in \N \mid \sum_{t=0}^{T-1} x_t\cdot\Delta s_{t+1} \ge 0,\ x_t\in D^\infty_t \text{ $P$-a.s.}\},
\]
so condition 2 in \thref{2lambda} holds.
\end{proof}

The above extends the existence results in Theorems~2.7 and 2.10 of \cite{rs5} by allowing for portfolio constraints and more general utility functions.

\begin{remark}\thlabel{rem:fmlin}
In the absence of portfolio constraints, condition 2 in \thref{thm:alm1} becomes~\eqref{na}.
\end{remark}

\begin{remark}\thlabel{rem:fmlb}
Condition 1 in \thref{thm:alm1} holds, in particular, if $V$ is lower bounded since then, $EV^*(0)<\infty$ so one can simply take $y=0$. More generally, the condition holds if there exists a $y\in L^1$ such that
\[
cy,\ \sigma_{D_t}(E_t[y\Delta s_{t+1}]),\ V^*(y)\in L^1\quad\forall t
\]
and one of the following conditions hold
\begin{enumerate}
\item[A]
  There exist $\lambda\in(0,1)$, $\bar y\in\dom EV^*$ and $C>0$ such that
  \[
  V^*(\lambda y,\omega)\le CV^*(y,\omega)\quad\forall y\in[0,\bar y(\omega)].
  \]
\item[B]
  There exist $\lambda>1$, $\bar y\in\dom EV^*$ and $C>0$ such that
  \[
  V^*(\lambda y,\omega)\le CV^*(y,\omega)\quad\forall y\ge \bar y(\omega).
  \]
\end{enumerate}
If $V$ is deterministic, then condition A holds if $AE_-(V)<1$ while B holds if $AE_+(V)>1$.
\end{remark}

\begin{proof}
The first claim is obvious. As to the second, let $y\in\dom EV^*$. Under A,
\begin{align*}
EV^*(\lambda y) &= E\one_{\{y\le \bar y\}}V^*(\lambda y) + E\one_{\{y> \bar y\}}V^*(\lambda y)\\
&\le E\one_{\{y\le \bar y\}}V^*(\lambda y) + E\one_{\{y> \bar y\}}\max\{V^*(\lambda \bar y),V^*(y)\}\\
&\le E\one_{\{y\le \bar y\}}CV^*(y) + E\one_{\{y> \bar y\}}\max\{CV^*(\bar y),V^*(y)\},
\end{align*}
where the first inequality comes from the convexity of $V^*$. Since $\bar y,y\in\dom EV^*$, the last expression is integrable. Under B,
\begin{align*}
V^*(\lambda y) &= \one_{\{y\le \bar y\}}V^*(\lambda y) + \one_{\{y> \bar y\}}V^*(\lambda y)\\
&\le \one_{\{y\le \bar y\}}\max\{V^*(y),V^*(\lambda \bar y)\} + \one_{\{y> \bar y\}}V^*(\lambda y)\\
&\le \one_{\{y\le \bar y\}}\max\{V^*(y),CV^*(\bar y)\} + \one_{\{y> \bar y\}}CV^*(y),
\end{align*}
where the first inequality comes from the convexity of $V^*$. Thus, under both A and B, $\lambda\dom EV^*\subseteq\dom EV^*$ so condition 1 in \thref{thm:alm1} holds.

We have $AE_-(V)<p$ if and only if there exists $\bar u<0$ such that 
\[
V'(u)\le p V(u)/u \quad\forall u\le \bar u,
\]
while $AE_+(V)>p$ if and only if there exists $\bar u>0$ such that 
\[
V'(u)\ge p V(u)/u \quad\forall u\ge \bar u,
\]
Thus, the last claim follows from \cite[Lemmas~20 and 21]{pp18d}.
\end{proof}

When applied directly to \eqref{alm}, the generalized Bellman equations \eqref{be} do not provide much information about the solutions. We will thus reformulate \eqref{alm} as an optimal control problem and find that the optimal investment strategy at time $t$ depends on the past allocations only through the current level of wealth. If $s$ is componentwise almost surely nonzero, we can write \eqref{alm} as a stochastic control problem
\begin{equation*}
\begin{aligned}
&\minimize\quad &  EV&\left(c-X_T \right)\quad\ovr\quad (X,U)\in\N,\\
&\st &  X_0 &= 0\\
& & \Delta X_{t} &= R_{t}\cdot U_{t-1}\quad\forall t=1,\dots,T\\
& & U_t &\in\tilde D_t\quad\forall t=0,\dots,T
\end{aligned}
\end{equation*}
where $X_t$ is the wealth generated by the trading strategy up to time $t$, $R^j_t:=\Delta s^j_t/s^j_{t-1}$ is the rate of return on asset $j$, $U_t^j:=s^j_tx^j_t $ is the amount of cash invested in asset $j$ over the period $(t,t+1]$ and
\[
\tilde D_t(\omega)=\{U\in\reals^{J}\mid (U^j/s^j_t(\omega))_{j\in J}\in D_t(\omega)\}.
\]
This formulation can be extended as follows.

\begin{example}\thlabel{ex:fmdp}
Consider the problem
\begin{equation*}
\begin{aligned}
&\minimize\quad &  EV&\left(c-X_T \right)\quad\ovr\quad (X,U)\in\N,\\
&\st &  X_0 &= w\\
& & \Delta X_{t} &= R_{t}\cdot U_{t-1}\quad\forall t=1,\dots,T\\
& & (X_t,U_t) &\in\hat D_t\quad\forall t=0,\dots,T-1,
\end{aligned}
\end{equation*}
where $\hat D_t$ is an $\F_t$-measurable random set in $\reals\times\reals^J$ and $w\in L^0(\F_0)$ is a given initial wealth. The dependence of the portfolio constraint on the wealth $X_t$ has practical significance. For instance,
\[
\hat D=\{(X_t,U_t)\mid \one\cdot U_t \le \alpha X_t\}
\]
describes a capital requirement where at most the proportion $\alpha$ of the current wealth can be invested in the risky assets while the constraint
\[
\hat D=\{(X_t,U_t)\mid \one\cdot U_t = X_t\}
\]   
would require the whole wealth to be invested in the risky assets. Various combinations of the above e.g., with short-selling constraints, can be incorporated into the constraints $\hat D$.

The above model fits the control format with $N=1$, $M=|J|$, $A_t=0$, $B_t=R_t$, $W_t=0$ and
\begin{align*}
  L_T(X_T,U_T) &=V(c-X_T),\\
  L_t(X_t,U_t) &=\delta_{\hat D_t}(X_t,U_t),\\
  L_0(X_0,U_0) &=\delta_{\{w\}}(X_0)+\delta_{\hat D_0}(X_t,U_t).
\end{align*}
Assuming that $V$ and $c$ are $\F_T$-measurable, the dynamic programming equations \eqref{be_OC} can be written as 
\begin{align}\label{eq:ocfm2}
\begin{split}
J_T(X_T) &=V(c-X_T)\\
I_{t+1}(X_{t},U_{t}) &= J_{t+1}(X_{t}+R_{t+1}\cdot U_{t}) \quad t=T-1,\dots,0,\\
J_t(X_t) &= \inf_{U_t\in\reals^J} \{(E_t I_{t+1})(X_t,U_t)\mid (U_t,X_t)\in \hat D_t\}\quad t=T-1,\dots,1,\\
J_0(X_0) &= \delta_{\{w\}}(X_0) + \inf_{U_0\in\reals^J} \{(E_0I_1)(X_0,U_0) \mid (X_0,U_0)\in \hat D_0 \}.
\end{split}
\end{align}
The optimality conditions \eqref{eq:ocfm2} imply that the optimal portfolio $U_t$ depends on the past decisions only via $X_t$.
\end{example}



\begin{example}[Exponential utility]\thlabel{ex:exputil}
Consider the model of \thref{ex:fmdp} in the case where $V(u)=\exp(\rho u)/\rho$ and $\hat D_t(\omega)=\reals\times\tilde D_t(\omega)$. A simple induction argument in \eqref{eq:ocfm2} gives
\begin{align*}
J_t(X_t) &=  \alpha_tV(-X_t),
\end{align*}
where $\alpha_T=\exp(\rho c)$ and 
\[
\alpha_t=\inf_{U_t\in\tilde D_t} E_t[\alpha_{t+1}\exp(-\rho R_{t+1}\cdot U_t)].
\]
An adapted portfolio process $\bar U$ satisfies the optimality condition  if and only if $\bar U_t$ achieves the infimum above almost surely for all $t$. Note that the amounts of wealth $U_t$ invested in the risky assets do not depend on the level of wealth $X_t$.

Assuming further that each $R_t$ is independent of $\F_{t-1}$ and that $\tilde D_t$ and $c$ are deterministic, we see that $\alpha_t$ are constants and the cost-to-go functions $J_t$ are deterministic. If, in addition, if $R_t$ is normally distributed with mean $\mu_t$ and covariance $\Sigma_t$, then, by the properties of moment generating functions,
\begin{align*}
E_t[\alpha_{t+1}\exp(-\rho R_{t+1}\cdot U_t)] &=\alpha_{t+1}E[\exp(-\rho R_{t+1}\cdot U_t)]\\
& = \alpha_{t+1}\exp(-\rho U_t\cdot\mu_t+\frac{\rho^2}{2}U_t\cdot\Sigma_t U_t),
\end{align*}
so, when there are no portfolio constraints, it is optimal to take $U_t =\Sigma_t^{-1}\mu_t/\rho$. This is a classic finding in portfolio optimization going back to \cite{mer69}.
\end{example}




\section{Appendix}

The following is a reformulation of \cite[Theorem~9.2]{roc70a}.

\begin{lemma}\thlabel{lineality}
Let $f:\reals^n\times\reals^m\to\ereals$ be convex and assume that 
\[
N:=\{x\in\reals^{n} \mid f^\infty(x,0)\le 0\}
\]
is linear. Then 
\[
p(u):=\inf_{x\in\reals^n} f(x,u,\omega)
\]
is a convex lsc function and
\[
p^\infty(u) = \inf_{x\in\reals^{n}}f^\infty(x,u).
\]
Moreover, $\argmin f(\cdot,u)\cap N^\perp \ne\emptyset$ for all $u\in\reals^m$.
\end{lemma}
\begin{proof}
Since $N$ is a convex cone, the linearity condition means that
\[
f^\infty(x,u)\le 0,\ f^\infty(-x,-u) >0\quad \Rightarrow\quad A(x,u) \ne 0,
\]
where $A(x,u)=u$. Thus, by \cite[Theorem~9.2]{roc70a}, $p$ is a lsc convex function, the formula for $p^\infty$ is valid and the infimum in the definition of $p$ is attained. That the infimum is attained by $x\in N^\perp$ follows from \cite[Corollary 8.6.1]{roc70a}
\end{proof}

The following was used in the proof of \thref{thm:dp0} and some of its corollaries in Section~\ref{sec:dpapp}.

\begin{theorem}\thlabel{thm:ip}
Let $f$ be a convex normal integrand on $\reals^n\times\reals^m\times\Omega$ and assume that the set-valued mapping
\[
N(\omega)=\{x\in\reals^{n} \mid f^\infty(x,0,\omega)\le 0\}
\]
is linear-valued. Then 
\[
p(u,\omega):=\inf_{x\in\reals^n} f(x,u,\omega)
\]
is a normal integrand on $\reals^m\times\Omega$ and
\[
p^\infty(u,\omega) = \inf_{x\in\reals^{n}}f^\infty(x,u,\omega).
\]
Moreover, given a $u\in \L^0(\F)$, there is an $x\in \L^0(\F)$ with $x(\omega)\perp N(\omega)$ and
\[
p(u(\omega),\omega) = f(x(\omega),u(\omega),\omega).
\]
\end{theorem}

\begin{proof}
By \thref{lineality}, the linearity condition implies for every $\omega\in\Omega$, that the set
\[
S(\omega):= N(\omega)^\perp\cap\argmin_xf(x,u,\omega)
\]
is nonempty and that $p(\cdot,\omega)$ is a lower semicontinuous convex function with
\[
p^\infty(u,\omega)=\inf_{x\in\reals^n}f^\infty(x,u,\omega).
\]
By \cite[Proposition 14.47]{rw98}, the lower semicontinuity implies that $p$ is a normal integrand. By \cite[Exercise~14.54, Proposition~14.33, Proposition~14.11, Theorem~14.37 ]{rw98}, the set $S$ is measurable so, by \cite[Corollary~14.6]{rw98}, it admits a measurable selection $x$.
\end{proof}

The three lemmas below extend classical properties of expectations and conditional expectations to general, possibly nonintegrable random variables.

\begin{lemma}\thlabel{lem:Esum}
Given extended real-valued random variables $\xi_1$ and $\xi_2$, we have
\[
E[\xi_1+\xi_2]=E[\xi_1]+E[\xi_2]
\]
under any of the following:
\begin{enumerate}
\item\label{int1}
  $\xi_1^+,\xi_2^+\in L^1$ or $\xi_1^-,\xi_2^-\in L^1$.
  \item\label{int2}
  $\xi_1\in L^1$ or $\xi_2\in L^1$,
\item\label{int4}
  $\xi_1$ or $\xi_2$ is $\{0,+\infty\}$-valued.
\end{enumerate}
\end{lemma}

\begin{proof}
If both $\xi_1$ and $\xi_2$ are positive, the equation holds by monotone convergence theorem. Assume now that $\xi_1,\xi_2\in L^1$ and let $\xi=\xi_1+\xi_2$. We have $\xi_+-\xi_-=\xi_1^+-\xi_1^-+\xi_2^+-\xi_2^-$ so $\xi^++\xi_1^-+\xi_2^-=\xi^-+\xi_1^++\xi_2^+$ and
\[
E\xi^++E\xi_1^-+E\xi_2^-=E\xi^-+E\xi_1^++E\xi_2^+.
\]
Rearranging gives the equality. In general, 
\begin{align*}
  &(\xi_1+\xi_2)^+\le \xi_1^++\xi_2^+,\quad(\xi_1+\xi_2)^-\le \xi_1^-+\xi_2^-,\\
  &\xi_i^+ \le (\xi_i+\xi_j)^++\xi_j^-,\quad \xi_i^-\le (\xi_i+\xi_j)^- +\xi_j^+.
\end{align*}
If $\xi_1^+,\xi_2^+\in L^1$, the first inequality implies $(\xi_1+\xi_2)^+\in L^1$. If $\xi_i^-\notin L^1$, the last inequality gives $(\xi_i+\xi_j)^-\notin L^1$ so the equation holds. The sufficiency of $\xi_1^-,\xi_2^-\in L^1$ follows by an analogous argument. Assume now $\xi_1\in L^1$. If $\xi_2^+\in L^1$, the equation holds by \ref{int1} while if $\xi_2^+\notin L^1$, the second inequality gives $(\xi_1+\xi_2)^+\notin L^1$, so the equation holds again. Claim \ref{int4} is clear.
\end{proof}

\begin{lemma}\thlabel{lem:ce}
Let $\xi_1$ and $\xi_2$ be extended real-valued random variables.
\begin{enumerate}
\item
  If $\xi_1$ and $\xi_2$ are quasi-integrable and satisfy any of the conditions in Lemma~\ref{lem:Esum}, then $\xi_1+\xi_2$ is quasi-integrable and
  \[
  E^\G[\xi_1+\xi_2] = E^\G [\xi_1] + E^\G[\xi_2].
  \]
\item
  If $\xi_2$ and $(\xi_1\xi_2)$ are quasi-integrable, and $\xi_1$ is $\G$-measurable, then
  \[
  E^\G[\xi_1\xi_2] = \xi_1E^\G[\xi_2].
  \]
\end{enumerate}
\end{lemma}

\begin{proof}
Let $\alpha\in L^\infty_+(\G)$. In 1, Lemma~\ref{lem:Esum} gives
\[
E[\alpha(\xi_1+\xi_2)]=E[\alpha\xi_1]+E[\alpha\xi_2] = E[\alpha E^\G\xi_1]+E[\alpha E^\G\xi_2] = E[\alpha(E^\G\xi_1+E^\G\xi_2)].
\]
To prove 2, note first that the claim is clear if $\xi_1$ is bounded and either nonpositive or nonnegative. Let $\xi^\nu_1$ be the projection $\xi_1$ to the $\uball_\nu$. Since $\xi_1\xi_2$ is quasi-integrable, 1 gives
\begin{align*}
E^\G[\xi^\nu_1\xi_2] &= E^\G[1_{\{\xi^\nu_1\ge 0\}} \xi^\nu_1\xi_2] +  E^\G[1_{\{\xi^\nu_1< 0\}} \xi^\nu_1\xi_2]\\
&=1_{\{\xi^\nu_1\ge 0\}} \xi^\nu_1E^\G[\xi_2] + 1_{\{\xi^\nu_1< 0\}} \xi^\nu_1 E^\G[\xi_2]\\
&=\xi^\nu_1 E^\G[\xi_2].
\end{align*}
The last term converges almost surely to $\xi_1E^\G[\xi_2]$. On the other hand, by 1,
\[
E^\G[\xi_1^\nu\xi_2] = E^\G[(\xi_1^\nu\xi_2)^+ -(\xi_1^\nu\xi_2)^-] =E^\G[(\xi_1^\nu\xi_2)^+] + E^\G[-(\xi_1^\nu\xi_2)^-].
\]
Since both terms on the right are monotone in $\nu$ and one of them is bounded,
\[
\lim E^\G[\xi_1^\nu\xi_2] = \lim E^\G[(\xi_1^\nu\xi_2)^+] + \lim E^\G[-(\xi_1^\nu\xi_2)^-].
\]
By conditional monotone convergence (see, e.g., \cite[Theorem II.7.2]{shi96}),
\[
\lim E^\G[\xi_1\xi_2] = E^\G[(\xi_1\xi_2)^+] + E^\G[-(\xi_1\xi_2)^-],
\]
where, by 1, the right side equals $E^\G[\xi_1\xi_2]$.
\end{proof}

\begin{lemma}\thlabel{tp}
Assume that $\xi$ is quasi-integrable and that $\G'\subseteq \G\subset\F$. Then
\[
E^{\G'}\xi = E^{\G'}[E^{\G}\xi].
\]
and
\[
E\xi=E[E^\G \xi].
\]
\end{lemma}
\begin{proof}
The first claim follows directly from the definition. As to the second, apply \thref{lem:ce} and \thref{lem:Esum} to $\xi=\xi^+-\xi^-$.
\end{proof}

The following is essentially \cite[Proposition 6.6 and Corollary 6.7]{kal2}
\begin{lemma}\thlabel{lem:condind}
Given $\sigma$-algebras $\G'$, $\G$ and $\H$, the following are equivalent:
\begin{enumerate}
\item $\G'$ and $\G$ are $\H$-conditionally independent,
\item $E^\H[w' w] = E^\H[w'] E^\H[w]$ for every $w'\in L^1(\G')$ and $w\in L^\infty(\G)$,
\item $ E^{\G\vee\H} [w'] = E^\H [w']$ for every $w'\in L^1(\G')$. 
\end{enumerate} 
In this case, if $\G$ and $\H$ are independent, so too are $\G$ and $\G'$.
\end{lemma}

\begin{proof}
The first implies the second by the monotone class theorem. When 2 holds, we have, for any $w'\in L^1(\G')$, $A\in\G$ and $B\in \H$,
\begin{align*}
E[E^{\H}[w'] 1_{A\cap B} ] = E[E^\H[w' 1_A] 1_B ] = E[w' 1_A 1_B] = E[E^{\G\vee\H} [w'] 1_{A\cap B} ],
\end{align*}
and, by the monotone class theorem, this extends from sets of the form $A\cap B$ to any set in $\G\vee\H$. Thus 2  implies 3. Assuming 3, we have, for $A'\in \G'$ and $A\in \G$, 
\begin{align*}
E^\H[1_A 1_{A'}] &= E^\H[E^{\G\vee\H} [1_A 1_{A'}]]= E^\H[1_A E^{\G\vee \H} 1_{A'}]\\
& =  E^\H[1_A E^{\H} 1_{A'}]= E^\H [1_A] E^{\H}[1_{A'}],
\end{align*}
so 1 holds.

Assume now in addition that $\G$ and $\H$ are independent. Given any $A'\in\G'$ and $A\in\G$, we get
\begin{align*}
E[1_{A'} 1_A] &=E[E^\H[1_{A'} 1_A]]=E[E^\H[1_{A'}]E^\H[1_A]]\\
&= E[E^\H[1_{A'}]E[1_A]]=E[1_{A'}]E[1_A],
\end{align*}
which implies the independence if $\G$ and $\G'$.
\end{proof}

\bibliographystyle{plain}
\bibliography{sp}

\end{document}